\documentclass{article}
\usepackage{amsmath}
\usepackage{amssymb}
\usepackage{amsthm}
\usepackage{dsfont}
\usepackage{footmisc}
\usepackage{pgfplots}
\usepgfplotslibrary{fillbetween}
\usetikzlibrary{decorations.markings}
\usetikzlibrary{decorations.pathreplacing}
\usetikzlibrary{patterns}
\usetikzlibrary{positioning}
\usepackage{graphics} 
\usepackage{epsfig} 
\usepackage{graphicx}
\usepackage{epstopdf}
\usepackage[plainpages=false, colorlinks=true, linkcolor=black,pdfpagelabels,bookmarksnumbered,bookmarksopen,urlcolor=black,citecolor=black]{hyperref}

\usepackage{authblk}

\newcommand\minus{%
  \setbox0=\hbox{-}%
  \vcenter{%
    \hrule width\wd0 height \the\fontdimen8\textfont3%
  }%
}

\newtheorem{theorem}{Theorem}[section]
\newtheorem{corollary}{Corollary}

\newtheorem{lemma}[theorem]{Lemma}

\theoremstyle{definition}

\newtheorem{remark}{Remark}

\newcommand{\unaryminus}{\scalebox{0.75}[1.0]{\( - \)}}

\DeclareMathOperator{\sign}{sign}

\title{Extended Hadamard expansions for the Airy functions}

\author[1]{Jose Luis Alvarez-Perez\thanks{joseluis.alvarez@uah.es}}
\affil[1]{Signal Theory and Communications Dept., Polytechnic
School, \ University of Alcala (Spain).}
\providecommand{\keywords}[1]{\textbf{\textit{Keywords---}} #1}
\providecommand{\href}[2]{#2}

\providecommand{\url}[1]{\texttt{#1}}

\date{}

\begin{document}
\maketitle

\begin{abstract}
There are two main power series for the Airy functions, namely the
Maclaurin and the asymptotic expansions. The former converges for
all finite values of the complex variable, $z$, but it requires a
large number of terms for large values of $|z|$, and the latter is a
Poincar\'{e}-type expansion which is well-suited for such large
values and where optimal truncation is possible. The asymptotic
series of the Airy function shows a classical example of the Stokes
phenomenon where a type of discontinuity occurs for the homonymous
multipliers. A new series expansion is presented here that stems
from the method of steepest descents and can be related to the
Hadamard expansions as presented in~\cite{paris:04,parisb:04}, and
which is convergent for all values of the complex variable. Hadamard
expansions were introduced as an extension of the method of steepest
descents and are defined in terms of a large number of
non-systematic integration path subdivisions. Unlike them, 
the expansions in the present work originate in the splitting of the
steepest descent in a number of segments that is not only finite but
very small, and which are defined on the basis of the location of
the branch points. One of the segments reaches to infinity and this
gives rise to the presence of upper incomplete Gamma functions. This
is one of the most important differences with the Hadamard series as
defined in the aforementioned references, where all the incomplete
Gamma functions are of the lower type. The theoretical interest of
the new series expansion is twofold. First of all, it shows how to
convert an asymptotic series into a convergent one with a finite
splitting of the steepest descent path. Secondly, the inverse of the
phase function that is part of the Laplace-type equation is
Taylor-expanded around branch points to produce Puiseux series when
necessary. In addition to this, the proposed analysis shows again
how the Stokes phenomenon for the Airy function is related to the
transition of the steepest descent paths at $\arg z = \pm 2 \pi/3$
from one to two. In regard to its computational application, these
series expansions require a relatively small number of terms for
each of them to reach a very high precision.

\end{abstract}

\section{Introduction}
\label{sec-introduction}
The Airy function was first introduced in 1838 for the calculation
of light intensity in a caustic, i.e. a surface where light is
focused after reflection or refraction by a curved interface between
media \cite{thorne:2017,landau:77,vallee:2004}. It is defined for
real values of $x$ by the following integral
\begin{equation}
\mbox{Ai}(x)=\frac{1}{\pi}\int_{0}^{\infty}du\,\cos(x
u+\frac{1}{3}u^{3})=\frac{1}{2\pi}\int_{-\infty}^{\infty}du\,e^{\,i(x
u+1/3 u^{3})}.
\label{eq-ai1}
\end{equation}
Although a number of other equivalent integral expressions that
result from elemental variable changes can be found in
\cite{abramowitz:72} and \cite{vallee:2004}, more insight is gained
if the integration is shown on the complex plane as in Figure
\ref{fig:paths} and written as
\begin{equation}
\mbox{Ai}(z)=\frac{1}{2\pi\,i}\int_{\mathcal{L}_{32}}du\,e^{z u-1/3
u^{3}}
\label{eq-ai2}
\end{equation}
where, additionally, a complex argument is considered, or,
alternatively, by using Cauchy's theorem,
\begin{equation}
\mbox{Ai}(z)=\frac{1}{2\pi\,i}\int_{\mathcal{L}_{31}+\mathcal{L}_{12}}du\,e^{\,z
u-1/3 u^{3}}.
\label{eq-ai3}
\end{equation}
\pgfplotsset{ compat=1.3, compat/path replacement=1.5.1, }
\begin{figure}
\resizebox{9cm}{!}{\begin{tikzpicture}
\pgfdeclarepatternformonly{north east lines wide}%
   {\pgfqpoint{-1pt}{-1pt}}%
   {\pgfqpoint{10pt}{10pt}}%
   {\pgfqpoint{9pt}{9pt}}%
   {
     \pgfsetlinewidth{0.4pt}
     \pgfpathmoveto{\pgfqpoint{0pt}{0pt}}
     \pgfpathlineto{\pgfqpoint{9.1pt}{9.1pt}}
     \pgfusepath{stroke}
    }
\pgfdeclarepatternformonly{north west lines wide}%
   {\pgfqpoint{-1pt}{1pt}}%
   {\pgfqpoint{10pt}{-10pt}}%
   {\pgfqpoint{9pt}{-9pt}}%
   {
     \pgfsetlinewidth{0.4pt}
     \pgfpathmoveto{\pgfqpoint{0pt}{0pt}}
     \pgfpathlineto{\pgfqpoint{9.1pt}{-9.1pt}}
     \pgfusepath{stroke}
    }
\begin{axis}[ticks=none,axis x line=center,axis y line=center,xlabel={$\operatorname{Re}(u)$}, ylabel={$\operatorname{Im}(u)$},
xlabel style={at={(axis description cs:1,0.5)}}, ylabel
style={at={(axis description cs:0.5,1)}}, ymin=-5,ymax=5],
\addplot
[dashed, black, domain=-5:5, samples=201, ] {0.57735*x} node
[pos=0.125, above, sloped, black] {$\theta=-\frac{5\pi}{6}$} node
[pos=0.9, below, sloped, black] {$\theta=\frac{\pi}{6}$};

\addplot [name path=A, black, opacity=0, domain=-5:-4.9, samples=2, ] {-0.57735*x};
\addplot [name path=B, black, opacity=0, domain=-5:0, samples=201, ] {5};
\addplot [pattern = north west lines wide, pattern color = gray!75] fill between [of=A and B];

\addplot [name path=C, black, opacity=0, domain=3:5, samples=201, ] {-0.57735*x};
\addplot [name path=D, black, opacity=0, domain=3:5, samples=201, ] {0.57735*x};
\addplot [pattern = north east lines wide, pattern color = gray!75] fill between [of=C and D];

\addplot [name path=E, black, opacity=0, domain=-5:-4.9, samples=201, ] {0.57735*x};
\addplot [name path=F, black, opacity=0, domain=-5:0, samples=201, ] {-5};
\addplot [pattern = north east lines wide, pattern color = gray!75] fill between [of=E and F];

\addplot [dashed, black, domain=-5:5, samples=201, ] {-0.57735*x} node
[pos=0.125, below, sloped, black] {$\theta=\frac{5\pi}{6}$} node
[pos=0.85, above, sloped, black] {$\theta=-\frac{\pi}{6}$};

\addplot [
black, domain=-5:0, samples=201, postaction={decorate,
decoration={markings,
 mark=at position 0.9 with {\arrow[scale=1.5]{<};}}}
] {sqrt(3*(x^2-1))};
\addplot [
black, domain=-5:0, samples=201, ] {-sqrt(3*(x^2-1))};
\addplot [
black, domain=0.25:5, samples=201, postaction={decorate,
decoration={markings,
 mark=at position 0.15 with {\arrow[scale=1.5]{>};}}}]
({-(((x-1)*sqrt(2+x))/(sqrt(3)*sqrt(x)))},{x});
\addplot [
black, domain=1.3:5, samples=201, ] ({x},{0.25*exp(-(x-1.3)/4)});
\addplot [
black, domain=0.25:5, samples=201, postaction={decorate,
decoration={markings,
 mark=at position 0.15 with {\arrow[scale=1.5]{<};}}}]
({-(((x-1)*sqrt(2+x))/(sqrt(3)*sqrt(x)))},{-x}); \addplot [ black,
domain=1.3:5, samples=201, ] ({x},{-0.25*exp(-(x-1.3)/4)}); \node at
(axis cs:1.5,1.75) [anchor=north east] {$\mathcal{L}_{12}$}; \node
at (axis cs:1.5,-0.75) [anchor=north east] {$\mathcal{L}_{31}$};
\node at (axis cs:-2,3.5) [anchor=north east] {$\mathcal{L}_{32}$};
\node at (axis cs:3,4.5) [anchor=north east] {$\mathds{C}$};
\node[gray!95] at (axis cs:4.75,-0.25) [anchor=north east] {1};
\node[gray!95] at (axis cs:-3.5,4.5) [anchor=north east] {2};
\node[gray!95] at (axis cs:-3.5,-3.5) [anchor=north east] {3};
\end{axis}
\end{tikzpicture}}
\caption{Integration paths in \eqref{eq-ai2} and \eqref{eq-ai3}. The
path $\mathcal{L}_{21}$ in equation \eqref{eq-bi1} is the same as
$\mathcal{L}_{12}$ but reversed.}
\label{fig:paths}
\end{figure}

Integration paths $\mathcal{L}_{m,n}\;(m,n=1,2,3;m\neq n)$ in Figure
\ref{fig:paths} are subject to the condition of
$-\pi/6+2(k-1)\pi/3<\mbox{Arg } z<\pi/6+2(k-1)\pi/3$ in each zone
$k=1,2,3$ as $|u|\rightarrow\infty$, where $\mbox{Arg } z$ refers to
the argument of $z$ \footnote{For the principal value of the
argument of $z$ we will use $\mbox{arg }z$.}. Integration in the
complex plane facilitates the analytic continuation of the Airy
function as given in Equation \eqref{eq-ai1}, which produces the
corresponding entire function $\mbox{Ai}(z),z\in \mathds{C}$.
Furthermore, the Airy function is a solution of the differential
equation
\begin{equation}
d^{2}y/dz^{2}=z\,y
\label{eq-diffairy}
\end{equation}
as can be seen by direct substitution of~\eqref{eq-ai2} or
\eqref{eq-ai3} as a function of $z$ in \eqref{eq-diffairy}, or else
by solving the latter with
Laplace's method \cite{copson:65}. 
Airy equation is a second-order, ordinary differential equation, the
general solution of which is a linear combination of two linearly
independent solutions that form its functional basis at all points
in $\mathds{C}$. A second linearly independent solution to
\eqref{eq-diffairy}, known as the Airy function of the second kind,
$\mbox{Bi}(z)$, is chosen as
\begin{equation}
\mbox{Bi}(z)=\frac{1}{2\pi\,i}\int_{\mathcal{L}_{31}+\mathcal{L}_{21}}du\,e^{\,z
u-1/3 u^{3}}.
\label{eq-bi1}
\end{equation}

Any modified Bessel function of order $\pm 1/3$ , $Z_{\pm 1/3}(x)$,
produces a solution of the Airy equation of the type
$\sqrt{x}\,Z_{\pm 1/3}(\frac{2}{3}x^{3/2})$ if $x>0$. Likewise,
solutions to the classical Bessel differential equation of the same
order, $R_{\pm 1/3}(x)$, produce Airy equation solutions given by
$\sqrt{-x}\,R_{\pm 1/3}\left(\frac{2}{3}(-x)^{3/2}\right)$ if $x<0$.
In particular, $\mbox{Ai}(x)$ can be written in terms of Bessel
functions of order $\pm 1/3$ in $\xi=\frac{2}{3}x^{3/2}$ as
in~\cite{vallee:2004}
\begin{subequations}
\begin{align}
\mbox{Ai}(x)&=\sqrt{\frac{x}{3}}\left[I_{\unaryminus 1/3}(\xi)
-I_{1/3}(\xi)\right]=\frac{1}{\pi}\sqrt{\frac{x}{3}}\,K_{1/3}(\xi)
\label{eq-bess1}
\\
\mbox{Ai}(\unaryminus x)&=\sqrt{\frac{x}{3}}\left[J_{\unaryminus
1/3}(\xi)
+J_{1/3}(\xi)\right]=\sqrt{\frac{x}{3}}\,\operatorname{Re}\left\{e^{i\,
\pi/6}\,H_{1/3}(\xi)\right\}
\label{eq-bess2}
\end{align}
\end{subequations}
for $x>0$. Analytic continuation guarantees that these identities
hold for a complex variable, namely $z=x+i\,y$ in \eqref{eq-bess1}
and $z=-x+i\,y$ with $x>0$ in \eqref{eq-bess2}.

The role of Airy functions is prominent in the construction of
uniform asymptotic expansions for contour integrals in the complex
plane with coalescing saddle points \cite{felsen:94,wong:2001}, 
as well as in solutions of linear second-order ordinary differential
equations with a simple turning point \cite{olver:97}. 
A large number of applications have been developed in physics:
almost in every instance for which wave equations with turning
points are relevant \cite{vallee:2004}.

\section{Maclaurin and asymptotic power expansion of the Airy functions}

The expansions of $\mbox{Ai}(z)$ and $\mbox{Bi}(z)$ near the origin
are given by \cite{copson:65}~\footnote{\label{fn-aibi}Equation
\eqref{eq-taylorseriesbi} is not given in \cite{copson:65} or
\cite{vallee:2004} but results of applying
\begin{equation}
\mbox{Bi}(z)=e^{i\,\pi/6}\mbox{Ai}(z\,e^{i\,2\pi/3})+e^{-i\,\pi/6}\mbox{Ai}(z\,e^{-i\,2\pi/3})
\label{eq-biai}
\end{equation}
to \eqref{eq-taylorseriesai}.}
\begin{subequations}
\begin{align}
\mbox{Ai}(z)&=\frac{1}{3^{2/3}\pi}\sum_{n=0}^{\infty}\frac{\Gamma(\frac{n+1}{3})}{n!}
\sin\left\{\frac{2}{3}(n+1)\pi\right\}\,\left(3^{1/3}z\right)^{n}
\label{eq-taylorseriesai}
\\
\mbox{Bi}(z)&=\frac{2}{3^{2/3}\pi}\sum_{n=0}^{\infty}\frac{\Gamma(\frac{n+1}{3})}{n!}
\sin^{2}\left\{\frac{2}{3}(n+1)\pi\right\}\,\left(3^{1/3}z\right)^{n}.
\label{eq-taylorseriesbi}
\end{align}
\end{subequations}

These Maclaurin series converge for all finite values of $z$ but can
be oscillating and slowly convergent. For large $|z|$, asymptotic
expansions of $\mbox{Ai}(z)$ and $\mbox{Bi}(z)$ are preferred
instead \cite{erdelyi:56,copson:65,olver:2010,felsen:94,wong:2001}.
Asymptotic expansions are not convergent but, for a fixed number of
terms, approach the exact function as the variable approaches some
distinguished value that defines the asymptotic limit. In
particular, $\mbox{Ai}(z)$ can be approximated by the following
asymptotic power expansions \cite{erdelyi:56,copson:65,felsen:94}
for large variable values
\begin{subequations}
\begin{align}
\mbox{Ai}(z)&\sim F(z) ,\quad |z|\rightarrow \infty\quad (|\mbox{Arg
} z|<\pi)
\label{eq-aiae1}
\\
\mbox{Ai}(z)&\sim F(z)+e^{i\pi/2\,{\scriptsize\mbox{sign}\{\mbox{Arg
} z}\}}\,G(z),\quad |z|\rightarrow \infty\quad
(\frac{\pi}{3}<|\mbox{Arg } z|<\frac{5\pi}{3})
\label{eq-aiae2}
\\
F(z)&=\frac{1}{2\pi\,z^{1/4}}\,e^{-2/3\,z^{3/2}}\sum_{n=0}^{\infty}(-1)^{n}
\frac{\Gamma(3n+\frac{1}{2})}{3^{2n}(2n)!}z^{-3/2 n}
\label{eq-f}
\\
G(z)&=\frac{1}{2\pi\,z^{1/4}}\,e^{2/3\,z^{3/2}}\sum_{n=0}^{\infty}
\frac{\Gamma(3n+\frac{1}{2})}{3^{2n}(2n)!}z^{-3/2 n}.
\label{eq-g}
\end{align}
\label{eq-subeq1}
\end{subequations}
Equation \eqref{eq-aiae1} is often quoted as applicable to
$|\mbox{Arg } z|<\pi/3$ \cite{erdelyi:56,felsen:94} only. The series
shown in \eqref{eq-aiae1} is obtained by using Watson's lemma,
whereas the restriction given by $|\mbox{Arg } z|<\pi/3$ results
from the steepest descents method. The former is obviously more
general and encompasses the latter. Stokes phenomenon is present in
the asymptotic expressions above: the overlapping of the regions in
\eqref{eq-subeq1} may seem to be inconsistent. The reason for such
overlapping is twofold, on the one hand there is a branch cut
implicit in $z^{3/2}$ at $\mbox{Arg } z=\pi$ in $\eqref{eq-aiae1}$,
whereas the branch cut is placed at $\mbox{Arg } z=0$ for
$\eqref{eq-aiae2}$~\footnote{It can be generally placed at any angle
in the range $(-\pi/3,\pi/3)$}, and, on the other hand, $G(z)$ is
subdominant to the asymptotic expansion in \eqref{eq-aiae1} for
$|\mbox{Arg } z| \in (\pi/3,\pi)$ for the corresponding Riemann
surface. For the case of $\mbox{Arg } z=\pm\pi$, which is the
anti-Stokes line of the turning point of \eqref{eq-diffairy}, $z=0$,
we can write, for $x\in\mathds{R}^{+}$~\footnote{$\mbox{Ai}(-z)$
admits an analog expression for $|z|\rightarrow \infty,\;|\mbox{Arg
} z|<2/3\pi$ \cite{copson:65}. However, the dominant exponential in
the trigonometric functions produce the expansion given in
\eqref{eq-aiae1}.},
\begin{subequations}
\begin{align}
\mbox{Ai}(-x)&\sim P(x)
\sin\left(\frac{2}{3}x^{3/2}+\frac{\pi}{4}\right)-Q(x)
\cos\left(\frac{2}{3}x^{3/2}+\frac{\pi}{4}\right)
\label{eq-aiae3}
\\
P(x)&=\frac{1}{\pi\,z^{1/4}}\sum_{n=0}^{\infty}(-1)^{n}
\frac{\Gamma(6n+\frac{1}{2})}{3^{4n}(4n)!}x^{-3 n}
\label{eq-p}
\\
Q(x)&=\frac{1}{\pi\,z^{7/4}}\sum_{n=0}^{\infty}(-1)^{n}
\frac{\Gamma(6n+\frac{7}{2})}{3^{4n+2}(4n+2)!}x^{-3 n}.
\label{eq-q}
\end{align}
\label{eq-subeq2}
\end{subequations}
Similar expressions are found for $\mbox{Bi}(z)$
\cite{copson:65}~\footnote{They can be obtained by using the
identity in previous footnote \ref{fn-aibi}.
},
\begin{subequations}
\begin{align}
\mbox{Bi}(z)&\sim 2 G(z),\quad |z|\rightarrow \infty\quad
(|\mbox{Arg } z|<\frac{\pi}{3})
\label{eq-biae3}
\\
\mbox{Bi}(z)&\sim 2
G(z)+e^{i\pi/2\,{\scriptsize\mbox{sign}\{\mbox{Arg }
z}\}}\,F(z),\quad |z|\rightarrow \infty\quad
(\frac{\pi}{3}<|\mbox{Arg } z|\leq\pi)
\label{eq-biae4}
\\
\mbox{Bi}(-x)&\sim P(x)
\cos\left(\frac{2}{3}x^{3/2}+\frac{\pi}{4}\right)+Q(x)
\sin\left(\frac{2}{3}x^{3/2}+\frac{\pi}{4}\right),\quad
x\in\mathds{R}^{+}.
\label{eq-biae5}
\end{align}
\label{eq-subeq3}
\end{subequations}
In \eqref{eq-subeq2} and \eqref{eq-biae5} the oscillatory nature of
the Airy functions for real negative values of the variable is
revealed.

The asymptotic series must be truncated at some term to produce an
acceptable value and its error is bound in magnitude by the first
neglected term in absolute value multiplied by a certain factor
\cite{olver:2010}. There are also some exponentially-improved
asymptotic series obtained by further expansion of the remainder
term \cite{olver:91}. To overcome the inherent difficulties of a
divergent series, techniques such as the Borel summation can be
applied to produce a hyperasymptotic series \cite{berry:90,boyd:99}.
Notwithstanding the well spread acceptance of these schemes and
their extensions, a convergent series expansions alternative to
\eqref{eq-taylorseriesai} and \eqref{eq-taylorseriesbi} is proposed
in the present work. 
In section \ref{sec-reviewreal}, the method of steepest descents for
a real variable is reviewed in the manner needed in section
\ref{sec-convseries} for producing the new series expansion. Such a
new expansion is valid for complex variables and is written in terms
of incomplete Gamma functions, as in the method of Hadamard series
for steepest descents described in~\cite{paris:04,parisb:04}. Here,
the integration path is divided in five sections in which the
corresponding series present uniform convergence. Unlike
in~\cite{paris:04,parisb:04}, the expansion in two of them is done
around the point of the path at infinity. Another difference is the
use of a Puiseux series for two of the expansions.

\section{Review of the asymptotic expansion of $\mbox{Ai}(x)$ of a real and positive variable by the method of steepest
descents}
\label{sec-reviewreal}

For the purpose of the demonstration of Theorem \ref{th-airyseries},
which is the main result of this work, we will use the already known
method that produces the customary series given in \eqref{eq-subeq1}
for $x\in\mathds{R}^{+}$ by changing the integration variable in
\eqref{eq-ai1} through $\alpha=i\,x^{-1/2} u$, 
\begin{equation}
\mbox{Ai}(x)=\frac{x^{1/2}}{2\pi i}\int_{-i \infty}^{i
\infty}d\alpha\,e^{x^{3/2}( \alpha-1/3 \alpha^{3})}.
\label{eq-ai4}
\end{equation}
The analytic continuation of $\mbox{Ai}(x)$ to the complex domain is
not treated in this section but will be included in the next one for
the new series expansion. Equation \eqref{eq-ai4} has the form
\begin{align}
I(\sigma)&=\int_{\mathcal{C}} d\alpha\,F(\alpha)\,e^{\sigma f(\alpha)}\\
\sigma&=x^{3/2}\in\mathds{R}^{+}
\label{eq-spint}
\end{align}
and can be solved by applying Debye's method of steepest descents
\cite{copson:65}. For that purpose the argument of the exponential
can be written in terms of the real and imaginary components of
$\alpha$ as
\begin{align}
f(\alpha)&=f_{r}(\alpha)+i
f_{i}(\alpha)=\alpha-\frac{\alpha^{3}}{3}\nonumber
\\
&f_{r}(\alpha)=\alpha_{r}-\frac{\alpha_{r}^{3}}{3}+\alpha_{r}
\alpha_{i}^{2}\nonumber
\\
&f_{i}(\alpha)=\alpha_{i}+\frac{\alpha_{i}^{3}}{3}-\alpha_{r}^{2}
\alpha_{i}.
\end{align}
Along the steepest descent path passing through the critical point
$\alpha_{s}=-1$~\footnote{The equivalent saddle point for complex
values of the variable of the Airy functions will be used in the
next section, as far as the new developments presented here are
concerned.} and determined by $f_{i}(\alpha)=f_{i}(\alpha_{s})$,
that is, by
\begin{equation}
1+\frac{\alpha_{i}^{2}}{3}-\alpha_{r}^{2}=0
\label{fig-stdsp1}
\end{equation}
we can write
\begin{equation}
f(\alpha)=f(\alpha_{s})-s^{2},\quad s\in\mathds{R}
\label{eq-taylors}
\end{equation}
which corresponds to
\begin{align}
f_{r}(\alpha)&=f_{r}(\alpha_{s})-s^{2}\nonumber
\\
f_{i}(\alpha)&=f_{i}(\alpha_{s})=0.
\label{eq-fsalpha}
\end{align}
The change of variable to $s$ -given in \eqref{eq-taylors} as an
implicit function- results in
\begin{equation}
I(\sigma)=\exp[\sigma\,f(\alpha_{s})]\int_{-\infty}^{\infty}\Phi(s)\exp[-\sigma\,s^{2}]\,ds
\label{eq-taylorsb}
\end{equation}
where
\begin{equation}
\Phi(s)=\frac{d\alpha}{ds}\,F[\alpha(s)].
\label{eq-phis}
\end{equation}
For the Airy's integral, $F[\alpha(s)]=1$ and
\begin{equation}
\Phi(s)=\frac{d\alpha}{ds}.
\label{eq-phisai}
\end{equation}
The function $\alpha(s)$ can be written through its Maclaurin series
with the help of the Lagrange's inversion theorem
\cite{whittakerandwatson:27}, that produces
\begin{equation}
\alpha+1=\sum_{n=1}^{\infty}\frac{i^{n}}{n!}\frac{\Gamma\left(\frac{3
n}{2}-1\right)} {\Gamma\left(\frac{n}{2}\right)}\frac{1}{3^{n-1}}
s^{n}
\label{eq-tayloralpha}
\end{equation}
where the radius of convergence is given by
$|s|\leq\rho=\frac{2}{\sqrt{3}}$. Uniform convergence in this same
region, $|s|\leq\rho$, follows from the application of Weierstrass'
M-test to \eqref{eq-tayloralpha} and the convergence of the upper
bounding series
\begin{equation}
\sum_{n=1}^{\infty}\frac{\Gamma\left(\frac{3}{2}n-1\right)}
{n!\,\Gamma\left(\frac{n}{2}\right)}\frac{1}{3^{n-1}}
\left(\frac{2}{\sqrt{3}}\right)^{n}=2.
\label{eq-mtest}
\end{equation}

Equation \eqref{eq-taylors} is a complex algebraic curve that
results in a single-valued function $\alpha=\alpha(s)$ on a Riemann
surface of three sheets and four branch points, namely $s=0,\pm 2
i/\sqrt{3}, \infty$. It is usually not mentioned that the series in
\eqref{eq-tayloralpha} expands around a branch point. This issue is
discussed in the next section. The result of using
\eqref{eq-tayloralpha} in \eqref{eq-phisai} and then in
\eqref{eq-taylorsb} beyond such radius of convergence produces the
following asymptotic approximation of $I$ (cf. \eqref{eq-subeq1}),
\begin{equation}
I(\sigma)\sim \exp[-2/3
\sigma]\sum_{n=0}^{\infty}\frac{\Phi^{(2n)}(0)}{(2n)!}
\frac{\Gamma(n+1/2)}{\sigma^{n+1/2}}
\label{eq-series}
\end{equation}
where only even derivatives of $\Phi$,
\begin{equation}
\Phi^{(2n)}(0)=i^{2n+1}\frac{\Gamma\left(3 n+\frac{1}{2}\right)}
{\Gamma\left(n+\frac{1}{2}\right)}\frac{1}{3^{2n}}
\label{eq-phi0}
\end{equation}
are present due to
\begin{align}
\int_{-\infty}^{\infty}s^{2n+1}\exp[-\sigma\,s^{2}]\,ds&= 0\nonumber
\\
\int_{-\infty}^{\infty}s^{2n}\exp[-\sigma\,s^{2}]\,ds&=
\frac{\Gamma(n+1/2)}{\sigma^{n+1/2}}
\label{eq-gammaintegrals2}
\end{align}
with $n\in\mathds{N}$ and $\sigma\in\mathds{R}^{+}$.

The lack of convergence of \eqref{eq-tayloralpha} in the domain of
integration of the Airy's integral leads to the lack of convergence
in \eqref{eq-series}. This use of the series beyond its convergence
limits is a technique often used in asymptotics, as long as the
series can be truncated with a known error bound. The method of
steepest descents is not the only manner to obtain \eqref{eq-series}
with \eqref{eq-phi0}: Watson's lemma \cite{copson:65} can be applied
on the first integral expression for the Airy function in
\eqref{eq-ai1} with the Maclaurin series for the cosine\footnote{The
cosine series is convergent, unlike \eqref{eq-tayloralpha}, for the
whole integration domain. However, it is not uniformly convergent
and therefore, if integrated term by term, does not produce a
convergent series.}. However, the method of steepest descents is the
starting point for the development presented in the following
section.

\section{A convergent series expansion for $\mbox{Ai}(z)$ of complex variable by the method of steepest
descents} \label{sec-convseries}

For the case of complex arguments in \eqref{eq-ai2}, the change of
variable performed in the previous section is replaced by
$\alpha=i\,|z|^{-1/2}u$ where $z$ is the complex variable of the
Airy function. Thus,
\begin{align}
\mbox{Ai}(z)&=\frac{|z|^{1/2}}{2\pi i}\int_{-i \infty}^{i
\infty}d\alpha\,e^{|z|^{3/2}(
w\alpha-1/3 \alpha^{3})}\nonumber\\
w&=e^{i\,\varphi}\nonumber\\
\varphi&=\mbox{arg }z
\label{eq-aiz1}
\end{align}
where $\mbox{arg }z$ is the principal value of the argument of $z$.
The integration path can be deformed to a path of the type
${\mathcal{L}_{32}}$ or $\mathcal{L}_{31}+\mathcal{L}_{12}$ in
Figure \ref{fig:paths}, or any equivalent homotopic curve.

Aiming at the integration along the steepest descent path, we now
compute the saddle points of the argument of the exponential
function in \eqref{eq-aiz1} as a function of $\alpha$,
\begin{align}
f(w,\alpha)&=f_{r}(w,\alpha)+i
f_{i}(w,\alpha)=w\alpha-\frac{\alpha^{3}}{3}\nonumber
\\
&f_{r}(w,\alpha)=w_{r}\alpha_{r}-w_{i}\alpha_{i}-\frac{\alpha_{r}^{3}}{3}+\alpha_{r}
\alpha_{i}^{2}\nonumber
\\
&f_{i}(w,\alpha)=w_{r}\alpha_{i}+w_{i}\alpha_{r}+\frac{\alpha_{i}^{3}}{3}-\alpha_{r}^{2}
\alpha_{i}
\label{eq-taylorscomplex0}
\end{align}
where $w=w_{r}+i w_{i}$ and $\alpha=\alpha_{r}+i \alpha_{i}$. These
saddle points are given by $\alpha_{s}=\pm w^{1/2}$. The steepest
descent path is homotopic to $\mathcal{L}_{32}$ and crosses
$\alpha_{s}=- w^{1/2}$. By setting
\begin{equation}
f(w,\alpha)=f(w,\alpha_{s}=- w^{1/2})-s^{2}=\unaryminus
\frac{2}{3}w^{3/2}-s^{2},\quad s\in\mathds{R}
\label{eq-taylorscomplex}
\end{equation}
so that
\begin{align}
f_{r}(w,\alpha)&=f_{r}(w,\alpha_{s})-s^{2}=-\frac{2}{3}\cos\frac{3}{2}\varphi-s^{2}\nonumber
\\
f_{i}(w,\alpha)&=f_{i}(w,\alpha_{s})=-\frac{2}{3}\sin\frac{3}{2}\varphi
\label{eq-fsalphacomplex}
\end{align}
we can define the steepest descent path in terms of the parameter
$s$.

Different steepest paths are shown in Figure \ref{fig:intpaths} for
different values of $\mbox{arg }z$. The presence of branch points in
the algebraic curve given by \eqref{eq-taylorscomplex} precludes the
analyticity of $\alpha=\alpha(s)$ in the whole path. 
The method developed by R. B. Paris in
\cite{paris:04,parisb:04,paris:11} consists in a non-systematic
splitting of the integration wherein a large number of segments is
produced to achieve analytic continuation. This results in a sum of
series containing lower incomplete Gamma functions that Paris names
Hadamard expansions. The same idea of dividing the steepest descent
path in a set of segments is used in what follows, but with a
criterion of bordering the branch points at a certain distance so
that analytic continuation is optimally performed along such an
integration path. In one case a branch point is found to lie on the
steepest descent path or very close to it and a Puiseux series is
used\footnote{In fact, $s=0$ is also a branch point of the algebraic
curve defined by \eqref{eq-taylorscomplex0} and
\eqref{eq-taylorscomplex}, as explained later in
footnote~\footref{refnote}.}. In addition to this, expansions around
infinity are carried out so that a infinite number of segments as in
previously studied Hadamard expansions is avoided. The positions of
the branch points is indicated in the following remark.

\begin{figure}
\resizebox{12cm}{!}{
\begin{tikzpicture}[scale=2]
    \def\offset{0.2}
    \def\offsetb{-4.0}
    \def\offsetv{-2.0}
    \def\offsetvb{0.16}
    \def\offsetvc{-1.4}
    \def\labeloffset{-0.2}
    \coordinate (A1) at (0+\offset,0);
    \coordinate (A1b) at (-0.5+\offset+\labeloffset,0);
    \coordinate (B1) at (1.8+\offset,0);
    \coordinate (C1) at (1.9+\offset,0.35);
    \coordinate (D1) at (2.35+\offset,0.35);
    \coordinate (E1) at (2.45+\offset,0.7);
    \coordinate (F1) at (3.8+\offset,0.7);
    \coordinate (A2) at (0+\offset,0.35);
    \coordinate (A2b) at (-0.5+\offset+\labeloffset,0.35);
    \coordinate (B2) at (0.45+\offset,0.35);
    \coordinate (C2) at (0.55+\offset,0.7);
    \coordinate (D2) at (2.35+\offset,0.7);
    \coordinate (E2) at (2.45+\offset,0.35);
    \coordinate (F2) at (3.8+\offset,0.35);
    \coordinate (A3) at (0+\offset,0.7);
    \coordinate (A3a) at (0+\offset+\labeloffset,1.1);
    \coordinate (A3b) at (-0.5+\offset+\labeloffset,0.7);
    \coordinate (A3c) at (0+\offset+\labeloffset,0.9+\offsetv);
    \coordinate (B3) at (0.45+\offset,0.7);
    \coordinate (C3) at (0.55+\offset,0.35);
    \coordinate (D3) at (1.8+\offset,0.35);
    \coordinate (E3) at (1.9+\offset,0);
    \coordinate (F3) at (3.8+\offset,0);
    \coordinate (G1inf) at (3.8+\offset,0);
    \coordinate (G2inf) at (3.8175+\offset,0.35);
    \coordinate (G3inf) at (3.8+\offset,0.7);
    \coordinate (G13inf) at (3.9+\offset,0.7);
    \coordinate (G31inf) at (3.9+\offset,0);
    \coordinate (G22inf) at (3.91+\offset,0.35);
    \coordinate (H1) at (4.2+\offset,0);
    \coordinate (H2) at (4.2+\offset,0.35);
    \coordinate (H3) at (4.2+\offset,0.7);
    \coordinate (C3z) at (4.3+\offsetb,0.6+\offsetv);
    \coordinate (D3z) at (5.25+\offsetb,0.6+\offsetv);
    \coordinate (E3z) at (5.5+\offsetb,-0.2+\offsetv);
    \coordinate (F3z) at (6.5+\offsetb,-0.2+\offsetv);
    \coordinate (A1z) at (4.3+\offsetb,-0.2+\offsetv);
    \coordinate (A1bz) at (4.95+\offsetb,0+\offsetv);
    \coordinate (A1cz) at (4.95+\offsetb,-0.4+\offsetv);
    \coordinate (A1dz) at (4.45+\offsetb,-0.1+\offsetv);
    \coordinate (A1ez) at (4.45+\offsetb,-0.3+\offsetv);
    \coordinate (B1z) at (5.25+\offsetb,-0.2+\offsetv);
    \coordinate (C1z) at (5.5+\offsetb,0.6+\offsetv);
    \coordinate (C1bz) at (5.8+\offsetb,0.8+\offsetv);
    \coordinate (C1cz) at (5.8+\offsetb,0.4+\offsetv);
    \coordinate (D1z) at (6.5+\offsetb,0.6+\offsetv);
    \coordinate (D1bz) at (6.3+\offsetb,0.7+\offsetv);
    \coordinate (D1cz) at (6.3+\offsetb,0.5+\offsetv);
    \coordinate (SNP1) at (4.65+\offsetb,-0.12+\offsetv);
    \coordinate (SNP2) at (5.2+\offsetb,-0.12+\offsetv);
    \coordinate (SNP3) at (5.3+\offsetb,0.2+\offsetv);
    \coordinate (SNN1) at (5.365+\offsetb,0.36+\offsetv);
    \coordinate (SNN2) at (5.45+\offsetb,0.68+\offsetv);
    \coordinate (SNN3) at (6.15+\offsetb,0.68+\offsetv);
    \coordinate (SPN1) at (4.65+\offsetb,-0.12+\offsetv-\offsetvb);
    \coordinate (SPN2) at (5.3+\offsetb,-0.12+\offsetv-\offsetvb);
    \coordinate (SPN3) at (5.38+\offsetb,0.125+\offsetv-\offsetvb);
    \coordinate (SPP1) at (5.445+\offsetb,0.36+\offsetv-\offsetvb);
    \coordinate (SPP2) at (5.53+\offsetb,0.68+\offsetv-\offsetvb);
    \coordinate (SPP3) at (6.15+\offsetb,0.68+\offsetv-\offsetvb);
    \draw (A1) -- (B1) -- (C1) -- (D1) -- (E1) -- (F1);
    \node at (A3a) {a$)$};
    \node at (A3b) {Sheet 1 $(\rightarrow \alpha_{1})$};
    \node at (A3c) {b$)$};
    \draw (A2) -- (B2) -- (C2) -- (D2) -- (E2) -- (F2);
    \node at (A2b) {Sheet 2 $(\rightarrow \alpha_{2})$};
    \draw (A3) -- (B3) -- (C3) -- (D3) -- (E3) -- (F3);
    \node at (A1b) {Sheet 3 $(\rightarrow \alpha_{3})$};
    \draw[dotted,very thick] (G1inf) -- (G13inf);
    \draw[dotted,thick] (G2inf) -- (G22inf);
    \draw[dotted,very thick] (G3inf) -- (G31inf);
    \draw (F1) -- (H3);
    \draw (F3) -- (H1);
    \draw (F2) -- (H2);
    \draw[densely dashed] (0+\offset,-0.35) -- (4.2+\offset,-0.35);
    \draw (0.47+\offset,-0.3) -- (0.47+\offset,-0.4);
    \node at (0.45+\offset,-0.55) {$s=-\frac{2 i w^{3/4}}{\sqrt{3}}$};
    \draw (1.87+\offset,-0.3) -- (1.87+\offset,-0.4);
    \node at (1.82+\offset,-0.55) {$s=0$};
    \draw (2.4+\offset,-0.3) -- (2.4+\offset,-0.4);
    \node at (2.6+\offset,-0.55) {$s=\frac{2 i w^{3/4}}{\sqrt{3}}$};
    \draw (3.85+\offset,-0.3) -- (3.85+\offset,-0.4);
    \node at (3.85+\offset,-0.55) {$s=\infty$};
    \draw[dashdotted,thick] (1.85+\offset,0.175) ellipse (16pt and 8pt);
    \draw[dashdotted,->,thick] (1.8+\offset,-0.1) -- (1.0+\offset,0.275+\offsetvc);
    \draw (C3z) -- (D3z) -- (E3z) -- (F3z);
    \draw (A1z) -- (B1z) -- (C1z) -- (D1z);
    \draw[<-,dashed,thick] (SNP1) -- (SNP2) -- (SNP3);
    \draw[<-,densely dotted,very thick] (SPN1) -- (SPN2) -- (SPN3);
    \draw[->,dashed,thick] (SNN1) -- (SNN2) -- (SNN3);
    \draw[->,densely dotted,very thick] (SPP1) -- (SPP2) -- (SPP3);
    \node at (A1bz) {$s>0$};
    \node at (A1cz) {$s<0$};
    \node at (C1bz) {$s<0$};
    \node at (C1cz) {$s>0$};
    \node at (D1bz) {$-i$};
    \node at (D1cz) {$+i$};
    \node at (A1dz) {$-i$};
    \node at (A1ez) {$+i$};
    \filldraw [gray] (0.5+\offset,0.525) circle (0.5pt);
    \filldraw [gray] (1.85+\offset,0.175) circle (0.5pt);
    \filldraw [gray] (2.4+\offset,0.525) circle (0.5pt);

\end{tikzpicture}}
\caption{Sheets of the Riemann surface for the solutions in Remark
\ref{rem1}. a) The three-sheeted Riemann surface has three finite
branch points and the branch point at infinity. b) Detail of the
branch point at $s=0$. The series expansion in Lemma \ref{lema1}
runs from negative $s$ values on sheet 3, which corresponds to
images in the branch given by $\alpha_{3}(s)$, to positive values on
sheet 2, which has its image values in the branch given by
$\alpha_{2}(s)$, as indicated by the dotted line. If the minus sign
were selected in equation \eqref{eq-gn2}, then the sheets would be
traversed in the opposite direction (dashed line). The arrows
indicate the fact that the branches are expanded \textit{around} the
central point $s=0$ and not the eventual direction of the
integration through a path, which we will take from $s=-\infty$ to
$s=\infty$.}
\label{fig:sheets1}
\end{figure}

\begin{remark}
{\itshape The complex algebraic curve defined by
\eqref{eq-taylorscomplex} has three branch points in the finite
domain of the Riemann surface, whose positions can be dealt with to
provide analytic continuation of the multi-valued function
$\alpha=\alpha(s)$. These branch points are readily computed by
seeking the roots of the corresponding discriminant and are found at
$s=0,\pm \frac{2 i w^{3/2}}{\sqrt{3}}$. The branch point at $s=0$ is
of order two but the other ones are simple branch points. In
addition to them, the infinity is a branch point of order four.

As a straightforward consequence of Cardano's formulas for the
reduced cubic, the three branches correspond to the following
solutions
\begin{align}
\alpha_{1}(s)&=\frac{\xi(s,w)}{2^{1/3}}
+\frac{2^{1/3}\,w}{\xi(s,w)}\nonumber
\\
\alpha_{2}(s)&=\frac{e^{i 2 \pi/3}}
{2^{1/3}}\,\xi(s,w)+\frac{2^{1/3}\,\,e^{\unaryminus i 2
\pi/3}\,w}{\xi(s,w)} \nonumber
\\
\alpha_{3}(s)&=\frac{e^{\unaryminus i 2 \pi/3}}
{2^{1/3}}\,\xi(s,w)+\frac{2^{1/3}\,\,e^{i 2
\pi/3}\,w}{\xi(s,w)}\nonumber
\\
&\xi(s,w)=[3\, s^{2}+2\,w^{3/2}+\sqrt{3}\, \mu(s,w)]^{1/3}\nonumber
\\
&\mu(s,w)= \sqrt{3\,s^{4}+4\,w^{3/2}\,s^{2}}.
\label{eq-alphahatsa}
\end{align}
For real values of $s$, functions $\alpha_{2}(s)$ and
$\alpha_{3}(s)$ define the steepest-descent path through
$\alpha_{s}=-w^{1/2}$. The three sheets of the single-valued
description of the algebraic curve \eqref{eq-taylorscomplex} are
depicted in Figure \ref{fig:sheets1}a. 

Considering these branch points and the convergence disks that they
allow, the steepest descent path will be split in five segments,
defined by the points around which each series expansion is
computed: $s=0, \pm \frac{2}{\sqrt{3}}, \infty$.}
\label{rem1}
\end{remark}

An extension of \eqref{eq-tayloralpha} for $\alpha_{s}=-w^{1/2}$ is
provided by the following Lemma, which will apply to the computation
of the central section of the steepest descent path crossing the
saddle point at $\alpha_{s}$.
\begin{lemma}
The solution of $\alpha=\alpha(s)$ around $\alpha_{s}=-w^{1/2}$ in
the equation
\begin{equation}
w\alpha-\frac{\alpha^{3}}{3}=-\frac{2}{3}w^{3/2}-s^{2}
\label{eq-chvar1}
\end{equation}
is given by the series
\begin{equation}
\alpha+w^{1/2}=\sum_{n=1}^{\infty}\frac{i^{n}}{n!}\frac{\Gamma\left(\frac{3}{2}n-1\right)}
{\Gamma\left(\frac{n}{2}\right)}w^{-3 n/4+1/2}\frac{1}{3^{n-1}}
s^{n}
\label{eq-seriescomplex}
\end{equation}
which is uniformly convergent for $|s|\leq 2/\sqrt{3}$.
\label{lema1}
\end{lemma}
\begin{proof}[\bfseries Proof]
The power series of $\alpha$ as a function of $s$  can be calculated
from the function $s=h(\alpha)$ by applying Lagrange's inverse
theorem~\footnote{The application of Taylor's theorem directly to
the function $\alpha=h^{-1}(s)$  as described in equations
\eqref{eq-alphahatsa} is more cumbersome.}. 
Thus
\cite{whittakerandwatson:27},
\begin{equation}
\alpha=\alpha_{s}+\sum_{n=1}^{\infty}
g_{n}\frac{[s-h(\alpha_{s})]^{n}}{n!}
\label{eq-alphaseries}
\end{equation}
where
\begin{equation}
g_{n}=\lim_{\alpha\rightarrow\alpha_{s}}\left[\frac{d^{n-1}}{d\alpha^{n-1}}
\left(\frac{\alpha-\alpha_{s}}{h(\alpha)-h(\alpha_{s})}\right)^{n}\right].
\label{eq-gn}
\end{equation}
The computation of~\eqref{eq-gn} for $h(\alpha)=\pm
i(\alpha+w^{1/2})\sqrt{w^{1/2}-\frac{1}{3}(\alpha+w^{1/2})}$ that
defines the steepest descent path through $\alpha_{s}=-w^{1/2}$
produces the result
\begin{equation}
g_{n}=(\mp
i)^{n}\frac{\Gamma\left(\frac{3}{2}n-1\right)}{\Gamma\left(\frac{n}{2}\right)}
\left(\frac{1}{3}\right)^{n-1} w^{-3 n/4+1/2}.
\label{eq-gn2}
\end{equation}
The positive sign will be chosen for the integration between
$s=-\infty$ to $s=\infty$ to follow the integration path sense shown
in Figure \ref{fig:paths}. Indeed, the choice of a positive sign in
\eqref{eq-gn2} implies that $s<0$ for the lower half-plane part of
the contour $\mathcal{L}_{32}$ and $s>0$ for its upper half-plane
part. Choosing the negative sign would reverse this correspondence.
This is shown in Figure \ref{fig:sheets1}b, together with the fact
that this series expands around the ramification point at $s=0$
reaching over $\alpha_{2}(s)$ and $\alpha_{3}(s)$. Hence,
equation~\eqref{eq-alphaseries} becomes
\begin{equation}
\alpha+w^{1/2}=\sum_{n=1}^{\infty}\frac{i^{n}}{n!}\frac{\Gamma\left(\frac{3}{2}n-1\right)}
{\Gamma\left(\frac{n}{2}\right)}w^{-3 n/4+1/2}\frac{1}{3^{n-1}}
s^{n}.
\label{eq-gnb}
\end{equation}
Equation~\eqref{eq-gnb} leads to equation \eqref{eq-tayloralpha}
when $w=1$. The upper bounding series in \eqref{eq-mtest} can also
be used for the Weierstrass' M-test on the uniform convergence of
\eqref{eq-gnb} in the disk $|s|\leq\frac{2}{\sqrt{3}}$.

\end{proof}


The term-by-term derivative of $\alpha$ with respect to $s$ in the
series \eqref{eq-seriescomplex} produces a uniform convergent series
on every compact subset of the disk $|s|<\frac{2}{\sqrt{3}}$ that
does not reach its boundary~\cite[pp.326-28, v.I]{markushevich:65}.
Therefore, when integrating the differentiated series we must do it
in compact subsets of the path section defined by $s \in
(\unaryminus\frac{2}{\sqrt{3}},\frac{2}{\sqrt{3}})$, which forms the
central section of the integration path. The next two integration
sections beyond this one are constructed to use the power expansions
around $s=\pm 2/\sqrt{3}$. For the case of $\mbox{\textup{arg} }z =
2\pi/3$, the point at $s=2/\sqrt{3}$ corresponds to a branch point
and a Puiseux series is obtained. Indeed, this would have been also
the case in Lemma \ref{lema1} should $s$ have been used instead of
$s^{2}$ in equation \eqref{eq-chvar1}~\footnote{The choice of $s^2$
instead of $s$ in \eqref{eq-chvar1} is motivated by the fact that
two sheets result from the branching at $s=0$, and it is also
convenient for the steepest descents method. The series in
\eqref{eq-seriescomplex} is thus sensitive to the sign of $s$
although the algebraic curve in \eqref{eq-chvar1} is not. Had $s$
been used instead of $s^{2}$, this expansion would have taken the
shape of a Puiseux series.\label{refnote}}. 
Furthermore, when $\mbox{\textup{arg} }z$ approaches $2\pi/3$, the
radius of convergence of the power series around $s=2/\sqrt{3}$
tends to zero and the expansion must be made around the branch point
outside the path for better convergence. With the next two lemmas,
the necessary series for representing the path in the neighborhoods
of $s=\pm 2/\sqrt{3}$ are given, after replacing the term $s^{2}$ by
$\frac{4}{3}+t$ in \eqref{eq-chvar1} so that the concerned expansion
is done around $t=0$.

\begin{lemma}
The power series of $\alpha=\alpha(t)$ around $t=0$ for the solution
of the equation
\begin{equation}
w\alpha-\frac{\alpha^{3}}{3}=-\frac{2}{3}w^{3/2}-\frac{4}{3}-t
\label{eq-chvar2}
\end{equation}
for the case of $w\neq e^{\, i 2\pi/3}$ is given by $\alpha^{+}$ and
$\alpha^{-}$~\footnote{The superscript $+$ indicates hence forth
that the expansion for $\alpha$ is computed around a point in the
second quadrant of the complex plane, where as the superscript $-$
does so for the third quadrant. The same stands for the roots
$\alpha_{0}^{+}$ and $\alpha_{0}^{-}$} as below
\begin{align}
\alpha^{\pm}&=\alpha_{0}^{\pm}+\sum_{n=1}^{\infty}\frac{(-1)^{n-1}}{n!}3^{n}
\,(\alpha_{0}^{\pm}-\alpha_{0}^{\mp})^{-2
n+1}\,(\alpha_{0}^{\pm}-\alpha_{0})^{-n}\,t^{n}\sum_{k=0}^{n-1}\binom{n-1}{k}
\nonumber
\\
&\times \frac{\Gamma(2 n-1-k)}{\Gamma(n)}
\frac{\Gamma(n+k)}{\Gamma(n)}
(\alpha_{0}^{\pm}-\alpha_{0}^{\mp})^{k}\,(\alpha_{0}^{\pm}-\alpha_{0})^{-k}
\label{eq-seriescomplex2}
\end{align}
where
\begin{align}
\alpha_{0}&=\frac{w}{\left(2+w^{3/2}+2\sqrt{1+w^{3/2}}\right)^{1/3}}+\left(2+w^{3/2}+2\sqrt{1+w^{3/2}}\right)^{1/3}\nonumber
\\
\alpha_{0}^{+}&=-\frac{(1+i\sqrt{3})w}{2\left(2+w^{3/2}+2\sqrt{1+w^{3/2}}\right)^{1/3}}
-\frac{1}{2}(1-i\sqrt{3})\left(2+w^{3/2}+2\sqrt{1+w^{3/2}}\right)^{1/3}\nonumber
\\
\alpha_{0}^{-}&=-\frac{(1-i\sqrt{3})w}{2\left(2+w^{3/2}+2\sqrt{1+w^{3/2}}\right)^{1/3}}
-\frac{1}{2}(1+i\sqrt{3})\left(2+w^{3/2}+2\sqrt{1+w^{3/2}}\right)^{1/3}.
\label{eq-alphas0pm}
\end{align}
The radius of convergence is given by
\begin{align}
r^{-\sign\{\arg z\}}&=\dfrac{4}{3}\nonumber \\
r^{+\sign\{\arg
z\}}&=\min\left\{\dfrac{4}{3},\dfrac{4}{3}\left|1+w^{3/2}\right|\right\}
\end{align}
for the corresponding superscripted $\alpha$ series in
\eqref{eq-seriescomplex2}.
\label{lema1b}
\end{lemma}
\begin{proof}[\bfseries Proof]
Equation \eqref{eq-chvar2} can be seen as $t=h(\alpha)$ where, as in
Lemma \ref{lema1}, Lagrange's inverse theorem can be applied. Thus,
we write
\begin{equation}
h(\alpha)=\frac{1}{3}(\alpha-\alpha_{0})(\alpha-\alpha_{0}^{+})(\alpha-\alpha_{0}^{-})
\label{eq-ht}
\end{equation}
where $\alpha_{0}, \alpha_{0}^{+}$ and $\alpha_{0}^{-}$ are given in
\eqref{eq-alphas0pm}. If $w\neq e^{\, i 2\pi/3}$, the three roots
are distinct. By applying equations \eqref{eq-alphaseries} and
\eqref{eq-gn} and the Leibniz rule for product differentiation, we
obtain
\begin{align}
g_{n}&=(-1)^{n-1}\,3^{n}\sum_{k=0}^{n-1}\binom{n-1}{k}
\frac{\Gamma(2 n-1-k)}{\Gamma(n)}\nonumber
\\
&\times \frac{\Gamma(n+k)}{\Gamma(n)}
(\alpha_{0}^{\pm}-\alpha_{0}^{\mp})^{-2
n+k+1}\,(\alpha_{0}^{\pm}-\alpha_{0})^{-n-k}
\end{align}
This results in \eqref{eq-seriescomplex2}. The branch points of
\eqref{eq-chvar2} are $t=-4/3,-4/3\,(1+w^{3/2}), \infty$. As there
are no singular points in \eqref{eq-chvar2}, and by considering
Figure \ref{fig:sheets2}a, it is straightforward to find that the
convergence radius of \eqref{eq-seriescomplex2} is $r^{+\sign\{\arg
z\}}=\min\left\{\frac{4}{3},\frac{4}{3}\left|1+w^{3/2}\right|\right\}$
and $r^{-\sign\{\arg z\}}=\frac{4}{3}$ for $\alpha^{+}$ and
$\alpha^{-}$.
\end{proof}

\begin{figure}
\resizebox{12cm}{!}{
\begin{tikzpicture}[scale=2]
    \def\offset{0.2}
    \def\offsetb{3.25}
    \def\offsetc{0.75}
    \def\offsetv{0.0}
    \def\offsetvb{0.16}
    \def\offsetvc{-1.4}
    \def\labeloffset{0}
    \coordinate (A1) at (0+\offset,0);
    \coordinate (A1b) at (-0.5+\offset+\labeloffset,0);
    \coordinate (B1) at (1.8+\offset,0);
    \coordinate (C1) at (1.9+\offset,0.35);
    \coordinate (D1) at (2.55+\offset,0.35);
    \coordinate (A2) at (0+\offset,0.35);
    \coordinate (A2b) at (-0.5+\offset+\labeloffset,0.35);
    \coordinate (B2) at (0.45+\offset,0.35);
    \coordinate (C2) at (0.55+\offset,0.7);
    \coordinate (D2) at (2.55+\offset,0.7);
    \coordinate (A3) at (0+\offset,0.7);
    \coordinate (A3a) at (-0.5+\offset+\labeloffset,1.1);
    \coordinate (A3b) at (-0.5+\offset+\labeloffset,0.7);
    \coordinate (A3c) at (0+\offsetb+\labeloffset,1.1+\offsetv);
    \coordinate (A4b) at (-0.5+\offset+\labeloffset,-0.95);
    \coordinate (A5b) at (-0.5+\offset+\labeloffset,-1.35);
    \coordinate (B3) at (0.45+\offset,0.7);
    \coordinate (C3) at (0.55+\offset,0.35);
    \coordinate (D3) at (1.8+\offset,0.35);
    \coordinate (E3) at (1.9+\offset,0);
    \coordinate (F3) at (2.55+\offset,0);
    \coordinate (G1inf) at (2.3+\offset,0);
    \coordinate (G2inf) at (2.3175+\offset,0.35);
    \coordinate (G3inf) at (2.3+\offset,0.7);
    \coordinate (G13inf) at (2.4+\offset,0.7);
    \coordinate (G31inf) at (2.4+\offset,0);
    \coordinate (G22inf) at (2.41+\offset,0.35);

    \coordinate (A1P) at (0+\offsetb,0);
    \coordinate (B1P) at (1.8+\offsetb-\offsetc,0);
    \coordinate (C1P) at (1.9+\offsetb-\offsetc,0.35);
    \coordinate (D1P) at (2.55+\offsetb-\offsetc,0.35);
    \coordinate (A2P) at (0+\offsetb,0.35);
    \coordinate (B2P) at (1.17+\offsetb-\offsetc,0.35);
    \coordinate (C2P) at (1.27+\offsetb-\offsetc,0.7);
    \coordinate (D2P) at (2.55+\offsetb-\offsetc,0.7);
    \coordinate (A3P) at (0+\offsetb,0.7);
    \coordinate (B3P) at (1.17+\offsetb-\offsetc,0.7);
    \coordinate (C3P) at (1.27+\offsetb-\offsetc,0.35);
    \coordinate (D3P) at (1.8+\offsetb-\offsetc,0.35);
    \coordinate (E3P) at (1.9+\offsetb-\offsetc,0);
    \coordinate (F3P) at (2.55+\offsetb-\offsetc,0);
    \coordinate (G1Pinf) at (2.3+\offsetb-\offsetc,0);
    \coordinate (G2Pinf) at (2.3175+\offsetb-\offsetc,0.35);
    \coordinate (G3Pinf) at (2.3+\offsetb-\offsetc,0.7);
    \coordinate (G13Pinf) at (2.4+\offsetb-\offsetc,0.7);
    \coordinate (G31Pinf) at (2.4+\offsetb-\offsetc,0);
    \coordinate (G22Pinf) at (2.41+\offsetb-\offsetc,0.35);

    \coordinate (S3m1) at (1.+\offset,0.035);
    \coordinate (S3m2) at (1.24+\offset,0.035);
    \coordinate (S3M1) at (1.30+\offset,0.035);
    \coordinate (S3M2) at (1.55+\offset,0.035);
    \coordinate (S2m1) at (1.+\offset,0.385);
    \coordinate (S2m2) at (1.24+\offset,0.385);
    \coordinate (S2M1) at (1.30+\offset,0.385);
    \coordinate (S2M2) at (1.55+\offset,0.385);

    \coordinate (ARROW3a) at (1.27+\offset,0.05);
    \coordinate (ARROW3b) at (1.335+\offset,0.15);
    \coordinate (ARROW3c) at (1.45+\offset,0.205);
    \coordinate (ARROW2a) at (1.27+\offset,0.4);
    \coordinate (ARROW2b) at (1.335+\offset,0.5);
    \coordinate (ARROW2c) at (1.45+\offset,0.555);
    \coordinate (ARROW1a) at (1.27+\offset,0.75);
    \coordinate (ARROW1b) at (1.335+\offset,0.85);
    \coordinate (ARROW1c) at (1.45+\offset,0.905);

    \coordinate (S3Pm1) at (0.95+\offsetb-\offsetc,0.035);
    \coordinate (S3Pm2) at (1.2+\offsetb-\offsetc,0.035);
    \coordinate (S3PM1) at (1.25+\offsetb-\offsetc,0.035);
    \coordinate (S3PM2) at (1.52+\offsetb-\offsetc,0.035);
    \coordinate (S12Pm1) at (0.95+\offsetb-\offsetc,0.385);
    \coordinate (S12Pm2) at (1.125+\offsetb-\offsetc,0.385);
    \coordinate (S12Pm3) at (1.175+\offsetb-\offsetc,0.525);
    \coordinate (S12PM1) at (1.3+\offsetb-\offsetc,0.385);
    \coordinate (S12PM2) at (1.4975+\offsetb-\offsetc,0.385);
    \coordinate (S12PM3) at (1.25+\offsetb-\offsetc,0.545);

    \coordinate (ARROW3Pa) at (1.27+\offsetb-\offsetc,0.05);
    \coordinate (ARROW3Pb) at (1.335+\offsetb-\offsetc,0.15);
    \coordinate (ARROW3Pc) at (1.45+\offsetb-\offsetc,0.205);
    \coordinate (ARROW12Pa) at (1.27+\offsetb-\offsetc,0.525);
    \coordinate (ARROW12Pb) at (1.335+\offsetb-\offsetc,0.525);
    \coordinate (ARROW12Pc) at (1.6+\offsetb-\offsetc,0.555);

    \draw (A1) -- (B1) -- (C1) -- (D1);
    \filldraw [gray] (0.5+\offset,0.525) circle (0.5pt);
    \node at (A3a) {a$)$};
    \node at (A3b) {Sheet 1};
    \node at (A3c) {b$)$};
    \draw (A2) -- (B2) -- (C2) -- (D2);
    \node at (A2b) {Sheet 2};
    \draw (A3) -- (B3) -- (C3) -- (D3) -- (E3) -- (F3);
    \filldraw [gray] (1.85+\offset,0.175) circle (0.5pt);
    \filldraw [gray] (1.27+\offset,0) circle (0.5pt);
    \filldraw [gray] (1.27+\offset,0.35) circle (0.5pt);
    \filldraw [gray] (1.27+\offset,0.7) circle (0.5pt);
    \node at (A1b) {Sheet 3};
    \draw[dotted,very thick] (G1inf) -- (G13inf);
    \draw[dotted,thick] (G2inf) -- (G22inf);
    \draw[dotted,very thick] (G3inf) -- (G31inf);
    \draw[densely dashed] (0+\offset,-0.35) -- (2.55+\offset,-0.35);%
    \draw (0.47+\offset,-0.3) -- (0.47+\offset,-0.4);
    \node at (0.5+\offset,-0.55) {$t{\scriptstyle =-{\textstyle \frac{4}{3}}(1+w^{3/2})}$};
    \draw (1.27+\offset,-0.3) -- (1.27+\offset,-0.4);
    \node at (1.24+\offset,-0.55) {$t{\scriptstyle }=0$};
    \draw (1.87+\offset,-0.3) -- (1.87+\offset,-0.4);
    \node at (1.8+\offset,-0.55) {$t{\scriptstyle =-{\textstyle \frac{4}{3}}}$};
    \draw (2.35+\offset,-0.3) -- (2.35+\offset,-0.4);
    \node at (2.35+\offset,-0.55) {$t{\scriptstyle}=\infty$};
    \draw[<-,densely dashed,thick] (S3m1) -- (S3m2);
    \draw[->,densely dashed,thick] (S3M1) -- (S3M2);
    \draw[->] (ARROW1a) -- (ARROW1b);
    \node at (ARROW1c) {${\scriptstyle \alpha_{\tiny 0}}$};
    \draw[->] (ARROW2a) -- (ARROW2b);
    \node at (ARROW2c) {${\scriptstyle \alpha_{\tiny 0}^{\tiny +}}$};
    \draw[->] (ARROW3a) -- (ARROW3b);
    \node at (ARROW3c) {${\scriptstyle \alpha_{\tiny 0}^{\tiny -}}$};
    \draw[<-,densely dashed,thick] (S2m1) -- (S2m2);
    \draw[->,densely dashed,thick] (S2M1) -- (S2M2);

    \draw (A1P) -- (B1P) -- (C1P) -- (D1P);
    \draw (A2P) -- (B2P) -- (C2P) -- (D2P);
    \draw (A3P) -- (B3P) -- (C3P) -- (D3P) -- (E3P) -- (F3P);
    \filldraw [gray] (1.22+\offsetb-\offsetc,0) circle (0.5pt);
    \filldraw [gray] (1.22+\offsetb-\offsetc,0.525) circle (0.5pt);
    \filldraw [gray] (1.85+\offsetb-\offsetc,0.175) circle (0.5pt);
    \draw[dotted,very thick] (G1Pinf) -- (G13Pinf);
    \draw[dotted,thick] (G2Pinf) -- (G22Pinf);
    \draw[dotted,very thick] (G3Pinf) -- (G31Pinf);
    \draw[densely dashed] (0+\offsetb,-0.35) -- (2.55+\offsetb-\offsetc,-0.35);
    \draw (1.27+\offsetb-\offsetc,-0.3) -- (1.27+\offsetb-\offsetc,-0.4);
    \node at (1.24+\offsetb-\offsetc,-0.55) {$t{\scriptstyle}=0$};
    \draw (1.87+\offsetb-\offsetc,-0.3) -- (1.87+\offsetb-\offsetc,-0.4);
    \node at (1.8+\offsetb-\offsetc,-0.55) {$t{\scriptstyle =-{\textstyle \frac{4}{3}}}$};
    \draw (2.35+\offsetb-\offsetc,-0.3) -- (2.35+\offsetb-\offsetc,-0.4);
    \node at (2.35+\offsetb-\offsetc,-0.55) {$t{\scriptstyle}=\infty$};
    \draw[<-,densely dashed,thick] (S3Pm1) -- (S3Pm2);
    \draw[->,densely dashed,thick] (S3PM1) -- (S3PM2);
    \draw[->,thin] (ARROW12Pa) -- (ARROW12Pb);
    \node at (ARROW12Pc) {${\scriptstyle \alpha_{\tiny 0}=\alpha_{\tiny 0}^{\tiny +}}$};
    \draw[->] (ARROW3Pa) -- (ARROW3Pb);
    \node at (ARROW3Pc) {${\scriptstyle \alpha_{\tiny 0}^{\tiny -}}$};
    \draw[<-,densely dashed,thick] (S12Pm1) -- (S12Pm2) -- (S12Pm3);
    \draw[->,densely dashed,thick] (S12PM3) -- (S12PM1) -- (S12PM2);

\end{tikzpicture}}
\caption{Sheets of the Riemann surface for the complex algebraic
curve given by \eqref{eq-chvar2} and $\arg z > 0$. a) This is the
case for $w\neq e^{\, i 2\pi/3}$. The series in Lemma \ref{lema1b}
are expansions around points $\alpha_{0}^{+}$ and $\alpha_{0}^{-}$.
The radius of the convergence disks are given by the distance to the
branch points at $t=-\frac{4}{3}$ and $t=-\frac{4}{3}(1+w^{3/2})$.
b) For the case $w= e^{\, i 2\pi/3}$, $\alpha_{0}=\alpha_{0}^{+}$ is
a branch point and we need a Puiseux series around it. The radius of
convergence is given then by $\rho=\frac{4}{3}$. If we were to deal
with the case of $\arg z > 0$, the labeling of the sheets would
change to follow the corresponding branches, but not the topology.}
\label{fig:sheets2}
\end{figure}

If $w = e^{\, i 2\pi/3}$, the roots in \eqref{eq-alphas0pm} are
$\alpha_{0}=\alpha_{0}^{+}=e^{\, i\pi/3}, \alpha_{0}^{-}=-2\,e^{\,
i\pi/3}$, and 
$t=0$ becomes a branch point of the algebraic curve in
\eqref{eq-chvar2}. The following Lemma deals with this case.
\begin{lemma}
The power series of $\alpha=\alpha(t)$ around $t=0$ for equation
\eqref{eq-chvar2} and $w=e^{\, i 2\pi/3}$ is given by $\alpha^{+}$
and $\alpha^{-}$ as
\begin{subequations}
\begin{align}
\alpha^{+}&=\alpha_{0}^{+}+\sum_{n=1}^{\infty}\frac{(-1)^{2
n-1}}{n!}3^{n/2}\,\frac{\Gamma(\frac{3
n}{2}-1)}{\Gamma(\frac{n}{2})}\,(\alpha_{0}^{+}-\alpha_{0}^{-})^{-3n/2+1}\,t^{n/2}\mid_{b}
\label{eq-seriescomplex3a}
\\
\alpha^{-}&=\alpha_{0}^{-}+\sum_{n=1}^{\infty}\frac{(-1)^{n-1}}{n!}3^{n}\,\frac{\Gamma(3
n-1)}{\Gamma(2 n)}\,(\alpha_{0}^{-}-\alpha_{0}^{+})^{-3n+1}\,t^{n}
\label{eq-seriescomplex3b}
\end{align}
\end{subequations}
with $\alpha_{0}^{+}=e^{\, i\pi/3}$ and $\alpha_{0}^{-}=-2\,e^{\,
i\pi/3}$ and where $\mid_{b}$ indicates the branch. The radius of
convergence is given by $|t|<\rho=\frac{4}{3}$. For the case of
$w=e^{\unaryminus i 2\pi/3}$, the superscripts $\pm$ must be
interchanged in both \eqref{eq-seriescomplex3a} and
\eqref{eq-seriescomplex3b} as well as taking
$\alpha_{0}^{+}=-2\,e^{\unaryminus i\pi/3}$ and
$\alpha_{0}^{-}=e^{\unaryminus i\pi/3}$. The branch noted by
$\mid_{b}$ is that of $(-1)^{1/2}=\pm i$ for $w=e^{\pm i 2\pi/3}$.
\label{lema1c}

\end{lemma}
\begin{proof}[\bfseries Proof]
The solutions of $h(\alpha)=0$ as defined in Lemma \ref{lema1b} and
for $w = e^{\, i \varphi}$ with $|\varphi|\leq 2\pi/3$ are all
simple roots of $h(\alpha)$ except for the case of $\varphi = \pm
2\pi/3$. In the case of $\varphi = 2\pi/3$, we have
\begin{equation}
h(\alpha)=\frac{1}{3}(\alpha-\alpha_{0}^{+})^{2}(\alpha-\alpha_{0}^{-})
\label{eq-ht2}
\end{equation}
with $\alpha_{0}^{+}=e^{\, i\pi/3}$ and $\alpha_{0}^{-}=-2\,e^{\,
i\pi/3}$. For $\alpha_{0}^{+}=e^{\, i\pi/3}$, $t=0$ is a branch
point of order 1, and it is convenient for an expansion around
$\alpha_{0}^{+}$ to write $t=h(\alpha)$ as
\begin{equation}
t^{1/2}=-\frac{1}{\sqrt{3}}(\alpha-\alpha_{0}^{+})\sqrt{\alpha-\alpha_{0}^{-}}
\label{eq-tbp}
\end{equation}
where the minus sign has been selected as the branch to integrate
within so that we remain at the steepest descent path. When $t<0$,
the branch for the square root function in \eqref{eq-tbp} is that of
$+i$ as $\alpha_{0}^{+}$ is in the second quadrant. Applying the
Lagrange inversion theorem for a power expansion around
$\alpha_{0}^{+}$ in terms of $t^{1/2}$, we obtain
\begin{align}
g_{n}&=(-1)^{2 n-1}\,3^{n/2}\frac{\Gamma(\frac{3
n}{2}-1)}{\Gamma(\frac{n}{2})}(\alpha_{0}^{+}-\alpha_{0}^{-})^{-3
n/2+1}
\end{align}
and the Puiseux series in equation \eqref{eq-seriescomplex3a}
results from it. 
\begin{figure}
\begin{tikzpicture}[scale=2]
    \def\offset{-1.35}
    \def\offsetb{0.05}
    \def\offsetc{-0.05}
    \def\offsetd{0.4}
    \def\offsetv{-0.3}
    \def\offsetvb{0.16}
    \def\offsetvc{-1.4}
    \def\offsetvd{0.175}
    \def\offsetvshat{-0.8}
    \def\offsetvt{-1.25}
    \def\offsetve{-2.6}
    \def\labeloffset{0}
    \coordinate (A1) at (0.2+\offset+\offsetb,0);
    \coordinate (B1) at (-0.1+\offset,\offsetv);
    \coordinate (C1) at (5.95+\offset,0);
    \coordinate (D1) at (0.2+\offset+\offsetb,0);
    \coordinate (D1b) at (1.525+\offset+\offsetb,0);
    \coordinate (D1b2) at (1.575+\offset+\offsetb,0);
    \coordinate (D1c) at (4.375+\offset+\offsetc,0);
    \coordinate (D1c2) at (4.425+\offset+\offsetc,0);
    \coordinate (D1d) at (6+\offset+\offsetc,0);
    \coordinate (D1d2) at (0.675+\offset+\offsetb,0);
    \coordinate (D1e) at (2.775+\offset+\offsetc,0);
    \coordinate (D1e2) at (3.775+\offset+\offsetb,0);
    \coordinate (D1f) at (4.95+\offset+\offsetc,0);
    \coordinate (E1) at (0.2+\offset,\offsetv);
    \coordinate (D2) at (0.7+\offset+\offsetb,0);
    \coordinate (E2) at (0.7+\offset,\offsetv);
    \coordinate (D3) at (1.55+\offset+\offsetb,0);
    \coordinate (E3) at (1.55+\offset,\offsetv);
    \coordinate (E3b) at (1.55+\offset,\offsetvd);
    \coordinate (D4) at (2.1+\offset+\offsetb,0);
    \coordinate (E4) at (2.1+\offset,\offsetv);
    \coordinate (D5) at (2.95+\offset+\offsetc,0);
    \coordinate (E5) at (2.9+\offset,\offsetv);
    \coordinate (E5b) at (2.85+\offset,\offsetc/2+\offsetvd);
    \coordinate (D6) at (3.55+\offset+\offsetc,0);
    \coordinate (E6) at (3.55+\offset,\offsetv);
    \coordinate (D7) at (4.4+\offset+\offsetc,0);
    \coordinate (E7) at (4.35+\offset,\offsetv);
    \coordinate (E7b) at (4.35+\offset,\offsetvd);
    \coordinate (D8) at (5.3+\offset+\offsetc,0);
    \coordinate (E8) at (5.15+\offset,\offsetv);
    \coordinate (D9) at (6+\offset+\offsetc,0);
    \coordinate (E9) at (5.95+\offset,\offsetv);
    \coordinate (E10) at (0+\offsetd,-3.0+\offsetve);
    \coordinate (F1) at (0.7+\offset+\offsetb,-0.05);
    \coordinate (F2) at (0.7+\offset+\offsetb,0.05);
    \coordinate (G1) at (2.1+\offset+\offsetb,-0.05);
    \coordinate (G2) at (2.1+\offset+\offsetb,0.05);
    \coordinate (H1) at (3.75+\offset+\offsetb,-0.05);
    \coordinate (H2) at (3.75+\offset+\offsetb,0.05);
    \coordinate (I1) at (4.95+\offset+\offsetb,-0.05);
    \coordinate (I2) at (4.95+\offset+\offsetb,0.05);

    \coordinate (B1t) at (-0.1+\offset,\offsetvshat);
    \coordinate (E1t) at (0.2+\offset,\offsetvshat);
    \coordinate (E2t) at (1.625+\offset,\offsetvshat);
    \coordinate (E8t) at (4.4+\offset,\offsetvshat);
    \coordinate (E9t) at (5.95+\offset,\offsetvshat);

    \coordinate (B1t2) at (-0.1+\offset,\offsetvt);
    \coordinate (E2t2) at (0.7+\offset,\offsetvt);
    \coordinate (E3t2) at (1.55+\offset,\offsetvt);
    \coordinate (E4t2) at (2.85+\offset,\offsetvt);
    \coordinate (E6t2) at (3.75+\offset,\offsetvt);
    \coordinate (E7t2) at (4.35+\offset,\offsetvt);
    \coordinate (E8t2) at (4.95+\offset,\offsetvt);

    \draw (A1) -- (C1);
    \node at (B1) {\scalebox{1.25}{$s:$}};
    \filldraw [gray] (D1) circle (1.pt);
    \node at (E1) {$-\infty$};
    \draw (F1) -- (F2);
    \filldraw [gray] (D3) circle (1.pt);
    \node at (E3) {$\unaryminus\dfrac{2}{\sqrt{3}}$};
    \node at (E3b) {$\alpha_{0}^{-}$};
    \filldraw [gray] (D5) circle (1.pt);
    \node at (E5) {$0$};
    \node at (E5b) {$\alpha_{s}$};
    \draw (H1) -- (H2);
    \filldraw [gray] (D7) circle (1.pt);
    \node at (E7) {$\dfrac{2}{\sqrt{3}}$};
    \node at (E7b) {$\alpha_{0}^{+}$};
    \draw (I1) -- (I2);
    \filldraw [gray] (D9) circle (1.pt);
    \node at (E9) {$\infty$};
    \node at (E10) {$\begin{array}{r l} m&=\left\{\begin{array}{cc}
                         \min\left\{\frac{4}{3},\frac{4}{3}\left|1+w^{3/2}\right|\right\} & \mbox{if } w\neq e^{\, i 2\pi/3}
                         \\ & \\
                         \frac{4}{3} & \mbox{if } w= e^{\, i 2\pi/3}
                       \end{array}
    \right.
    \\& \\
    s_{0}&=1\\& \\
    s_{1}&=\sqrt{\dfrac{5}{3}}\\& \\
    \hat{s}_{0}&=\left(\dfrac{5}{3}\right)^{-1/3}
    \\& \\ t_{0}&=\dfrac{1}{3}
    \end{array}$};

    \coordinate (B1t3) at (1.7+\offsetd,-4.25+\offsetve);
    \node at (B1t3){\framebox[30\totalheight]{$^{\ast}$ With the corresponding formulas in Lemmas \ref{lema1}
    and \ref{lema1b}}};

    \node at (B1t) {\scalebox{1.25}{$\hat{s}:$}};
    \node at (E1t) {$0$};
    \node at (E2t) {$\left(\dfrac{3}{4}\right)^{1/3}$};
    \node at (E8t) {$\left(\dfrac{3}{4}\right)^{1/3}$};
    \node at (E9t) {$0$};

    \node at (B1t2) {\scalebox{1.25}{$t:$}};
    \node at (E2t2) {$\dfrac{4}{3}$};
    \node at (E3t2) {$0$};
    \node at (E4t2) {$\unaryminus\dfrac{4}{3}$};
    \node at (E6t2) {$\unaryminus m$};
    \node at (E7t2) {$0$};
    \node at (E8t2) {$m$};
    \draw [decorate,decoration={brace,amplitude=5pt,raise=4ex}]
    (D1) -- (D1b) node[midway,yshift=3.5em]{\begin{tabular}{c} Convergence region \\ for $\hat{\alpha}_{3}(\hat{s})$\end{tabular}};
    \draw [decorate,decoration={brace,amplitude=5pt,raise=4ex}]
    (D1b2) -- (D1c) node[midway,yshift=3.5em]{\begin{tabular}{c} Convergence region \\ for  $\alpha(s)$\end{tabular}};
    \draw [decorate,decoration={brace,amplitude=5pt,raise=4ex}]
    (D1c2) -- (D1d) node[midway,yshift=3.5em]{\begin{tabular}{c} Convergence region \\ for   $\hat{\alpha}_{2}(\hat{s})$\end{tabular}};
    \draw [decorate,decoration={brace,amplitude=5pt,raise=12ex}]
    (D1d2) -- (D1e) node[midway,yshift=7em]{\begin{tabular}{c} Convergence region \\ for $\alpha^{-}(t)^{\ast}$\end{tabular}};
    \draw [decorate,decoration={brace,amplitude=5pt,raise=12ex}]
    (D1e2) -- (D1f) node[midway,yshift=7em]{\begin{tabular}{c} Convergence region \\ for $\alpha^{+}(t)^{\ast}$\end{tabular}};

    \def\offset{-1.35}
    \def\offsetb{0.05}
    \def\offsetc{-0.05}
    \def\offsetd{0.4}
    \def\offsetv{-0.3}
    \def\offsetvb{0.16}
    \def\offsetvc{-1.4}
    \def\offsetvd{0.175}
    \def\offsetvshat{-0.8}
    \def\offsetvt{-1.25}
    \def\labeloffset{0}
    \coordinate (A1) at (0.2+\offset+\offsetb,\offsetve);
    \coordinate (B1) at (-0.1+\offset,\offsetv+\offsetve);
    \coordinate (C1) at (5.95+\offset,\offsetve);
    \coordinate (D1) at (0.2+\offset+\offsetb,\offsetve);
    \coordinate (D1b) at (0.975+\offset+\offsetb,\offsetve);
    \coordinate (D1b2) at (1.025+\offset+\offsetb,\offsetve);
    \coordinate (D1c) at (2.075+\offset+\offsetb,\offsetve);
    \coordinate (D1c2) at (2.125+\offset+\offsetb,\offsetve);
    \coordinate (D1d) at (3.525+\offset+\offsetb,\offsetve);
    \coordinate (D1d2) at (3.575+\offset+\offsetb,\offsetve);
    \coordinate (D1e) at (5.025+\offset+\offsetb,\offsetve);
    \coordinate (D1e2) at (5.075+\offset+\offsetb,\offsetve);
    \coordinate (D1f) at (5.975+\offset+\offsetc,\offsetve);
    \coordinate (E1) at (0.2+\offset,\offsetv+\offsetve);
    \coordinate (D2) at (1+\offset+\offsetb,\offsetve);
    \coordinate (E2) at (1+\offset,\offsetv+\offsetve);
    \coordinate (D3) at (1.55+\offset+\offsetb,\offsetve);
    \coordinate (E3) at (1.1+\offset,\offsetv+\offsetve);
    \coordinate (E3b) at (1.55+\offset,\offsetvd+\offsetve);
    \coordinate (D4) at (2.1+\offset+\offsetb,\offsetve);
    \coordinate (E4) at (2.1+\offset,\offsetv+\offsetve);
    \coordinate (D5) at (2.95+\offset+\offsetc,\offsetve);
    \coordinate (E5) at (2.9+\offset,\offsetv+\offsetve);
    \coordinate (E5b) at (2.85+\offset,\offsetc/2+\offsetvd+\offsetve);
    \coordinate (D6) at (3.55+\offset+\offsetc,\offsetve);
    \coordinate (E6) at (3.6+\offset,\offsetv+\offsetve);
    \coordinate (D7) at (4.4+\offset+\offsetc,\offsetve);
    \coordinate (E7) at (5.1+\offset,\offsetv+\offsetve);
    \coordinate (E7b) at (4.35+\offset,\offsetvd+\offsetve);
    \coordinate (D8) at (5.3+\offset+\offsetc,\offsetve);
    \coordinate (E8) at (5.15+\offset,\offsetv+\offsetve);
    \coordinate (D9) at (6+\offset+\offsetc,\offsetve);
    \coordinate (E9) at (5.95+\offset,\offsetv+\offsetve);
    \coordinate (E10) at (0+\offsetd,-2.75+\offsetve);
    \coordinate (F1) at (1+\offset+\offsetb,-0.05+\offsetve);
    \coordinate (F2) at (1+\offset+\offsetb,0.05+\offsetve);
    \coordinate (G1) at (2.1+\offset+\offsetb,-0.05+\offsetve);
    \coordinate (G2) at (2.1+\offset+\offsetb,0.05+\offsetve);
    \coordinate (H1) at (3.55+\offset+\offsetb,-0.05+\offsetve);
    \coordinate (H2) at (3.55+\offset+\offsetb,0.05+\offsetve);
    \coordinate (I1) at (5.05+\offset+\offsetb,-0.05+\offsetve);
    \coordinate (I2) at (5.05+\offset+\offsetb,0.05+\offsetve);

    \coordinate (B1t) at (-0.1+\offset,\offsetvshat+\offsetve);
    \coordinate (E1t) at (0.2+\offset,\offsetvshat+\offsetve);
    \coordinate (E2t) at (1.1+\offset,\offsetvshat+\offsetve);
    \coordinate (E8t) at (5.1+\offset,\offsetvshat+\offsetve);
    \coordinate (E9t) at (5.95+\offset,\offsetvshat+\offsetve);

    \coordinate (B1t2) at (-0.1+\offset,\offsetvt+\offsetve);
    \coordinate (E2t2) at (1.1+\offset,\offsetvt+\offsetve);
    \coordinate (E3t2) at (1.4+\offset,\offsetvt+\offsetve);
    \coordinate (E4t2) at (2.1+\offset,\offsetvt+\offsetve);
    \coordinate (E6t2) at (3.55+\offset,\offsetvt+\offsetve);
    \coordinate (E7t2) at (4.35+\offset,\offsetvt+\offsetve);
    \coordinate (E8t2) at (5.1+\offset,\offsetvt+\offsetve);

    \draw (A1) -- (C1);
    \node at (B1) {\scalebox{1.25}{$s:$}};
    \filldraw [gray] (D1) circle (1.pt);
    \node at (E1) {$-\infty$};
    \draw (F1) -- (F2);
    \filldraw [gray] (D3) circle (1.pt);
    \node at (E3) {$\unaryminus s_{1}$};
    \node at (E3b) {$\alpha_{0}^{-}$};
    \node at (E4) {$\unaryminus s_{0}$};
    \draw (G1) -- (G2);
    \filldraw [gray] (D5) circle (1.pt);
    \node at (E5) {$0$};
    \node at (E5b) {$\alpha_{s}$};
    \node at (E6) {$s_{0}$};
    \draw (H1) -- (H2);
    \filldraw [gray] (D7) circle (1.pt);
    \node at (E7) {$s_{1}$};
    \node at (E7b) {$\alpha_{0}^{+}$};
    \draw (I1) -- (I2);
    \filldraw [gray] (D9) circle (1.pt);
    \node at (E9) {$\infty$};

    \node at (B1t) {\scalebox{1.25}{$\hat{s}:$}};
    \node at (E1t) {$0$};
    \node at (E2t) {$\hat{s}_{0}$};
    \node at (E8t) {$\hat{s}_{0}$};
    \node at (E9t) {$0$};

    \node at (B1t2) {\scalebox{1.25}{$t:$}};
    \node at (E2t2) {$t_{0}$};
    \node at (E4t2) {$\unaryminus t_{0}$};
    \node at (E6t2) {$\unaryminus t_{0}$};
    \node at (E8t2) {$t_{0}$};
    \draw [decorate,decoration={brace,amplitude=5pt,raise=4ex}]
    (D1) -- (D1b) node[midway,yshift=3em]{$\mathcal{L}_{-\infty}$};
    \draw [decorate,decoration={brace,amplitude=5pt,raise=4ex}]
    (D1b2) -- (D1c) node[midway,yshift=3em]{$\mathcal{L}_{\alpha_{0}^{-}}$};
    \draw [decorate,decoration={brace,amplitude=5pt,raise=4ex}]
    (D1c2) -- (D1d) node[midway,yshift=3em]{$\mathcal{L}_{s}$};
    \draw [decorate,decoration={brace,amplitude=5pt,raise=4ex}]
    (D1d2) -- (D1e) node[midway,yshift=3em]{$\mathcal{L}_{\alpha_{0}^{+}}$};
    \draw [decorate,decoration={brace,amplitude=5pt,raise=4ex}]
    (D1e2) -- (D1f) node[midway,yshift=3em]{$\mathcal{L}_{\infty}$};
\end{tikzpicture}
\caption{Integration segments. The regions of convergence for the
series expansions of Lemmas \ref{lema1} to \ref{lema1c} are
represented in the above part of the figure. The bottom part shows
the selection of segments for the integration segments as they are
defined to produce the addends in equation \eqref{eq-aiIexp}.}
\label{fig:convergenceregions}
\end{figure}

If the power expansion is done around $\alpha_{0}^{-}$, instead of
\eqref{eq-tbp} we start from
\begin{equation}
t=\frac{1}{3}(\alpha-\alpha_{0}^{+})^{2}(\alpha-\alpha_{0}^{-})
\label{eq-tbp2}
\end{equation}
to produce
\begin{align}
g_{n}&=(-1)^{n-1}\,3^{n}\frac{\Gamma(3 n-1}{\Gamma(2
n)}(\alpha_{0}^{-}-\alpha_{0}^{+})^{-3 n+1}
\end{align}
and thus obtain the expansion in \eqref{eq-seriescomplex3b}. In this
case, we have an expansion in powers of $t$ that are positive
integers. 

As shown in Figure \ref{fig:sheets2}b, there is a sheet of the
complex values of $t$ where $t=0$ is not a branch point and that
corresponds to the branch where $\alpha_{0}^{-}=\alpha(t=0)$,
whereas there is another one that contains both $t=0$ and $t=4/3$ as
branch points and where $\alpha_{0}^{+}=\alpha(t=0)$. The series in
\eqref{eq-seriescomplex3a} and \eqref{eq-seriescomplex3b}
correspond, respectively, to the power expansions around $t=0$ on
both sheets.

For the case of $\varphi =\unaryminus 2\pi/3$, the series for
$\alpha^{+}$ and $\alpha^{-}$ must be interchanged as well as the
values of $\alpha_{0}^{+}$ and $\alpha_{0}^{-}$, which must be also
replaced by their conjugates. The branch for $t^{n/2}$ is such that
$t<0$ produce negative imaginary values for odd values of $n$.
\end{proof}

When the distance of the branch point at $t=-4/3\,(1+w^{3/2})$ to
$t=0$ is equal to $2/3$, the circumference centred at this point
crosses the $t$ axis at $t=-2/3$ and $t=1/3$ (see Figure
\ref{fig:covergencedisks}). As $\arg z$ approaches $\pm 2\pi/3$, the
branch point approaches $t=0$ and it is more convenient to expand
$\alpha=\alpha(t)$ around the branch point for the segment $t\in
[-1/3,1/3]$ than using the expansions of Lemma \ref{lema1c}.

\begin{figure}
\begin{tikzpicture}[scale=2,cap=round]
    \begin{scope}[thick,font=\normalsize]
    \draw [->] (-2.75,0) -- (1.25,0) node [above left, xshift=1.5cm]  {$\scalebox{1.25}{Re\{t\}}$};
    \draw [->] (0,-2.1) -- (0,1.5) node [below right, yshift=0.25cm] {$\scalebox{1.25}{Im\{t\}}$};


    \path [draw=none,fill=black!25,semitransparent] (-4/3,0) circle (4/3);
    \draw[dashed] (-4/3,0) circle (4/3);
    \path [draw=none,fill=black!25,semitransparent] (-0.1667,-0.645) circle (4/3);
    \draw[dashed] (-0.1667,-0.645) circle (4/3);
    \path [draw=none,fill=black!30,semitransparent] (0,0) circle (2/3);
    \draw[dashed] (0,0) circle (2/3);
    \foreach \n in {-2.5,-2.0,-1.5,-1.0,-0.5,0.5,1.0}{%
        \draw (\n,-2pt) -- (\n,2pt)   node [above] {$\n$};
    }
    \foreach \n in {-2.25,-1.75,-1.25,-0.75,-0.5,-0.25,0.25,0.5,0.75}{%
        \draw (\n,-1pt) -- (\n,1pt);
    }
    \foreach \n in {-2.0,-1.5,-1.0,-0.5,0.5,1.0}{%
        \draw (-2pt,\n) -- (2pt,\n)   node [right] {$\n$};
    }
    \foreach \n in {-1.75,-1.25,-0.75,-0.25,0.25,0.75,1.25}{%
        \draw (-1pt,\n) -- (1pt,\n);
    }
    \path [draw=black,fill=none,semitransparent] (-0.1667,-0.645) circle (0.816);
    \path [draw=black,fill=none,semitransparent] (0,0) circle (1/3);
    \filldraw [black] (-4/3,0) circle (0.75pt);
    \filldraw [gray] (-2/3,0) circle (0.75pt);
    \filldraw [gray] (-1/3,0) circle (0.75pt);
    \filldraw [gray] (1/3,0) circle (0.75pt);
    \filldraw [black] (-0.1667,-0.645) circle (0.75pt);
    \node at (-4/3-0.1,-0.175) {$\scriptstyle t=-4/3$};
    \node at (-2/3-0.15,-0.175) {$\scriptstyle t=-2/3$};
    \node at (-1/3+0.05,-0.175) {$\scriptstyle t=-1/3$};
    \node at (1/3+0.025-0.1,-0.175) {$\scriptstyle t=1/3$};
    \node at (-0.1667-0.575,-0.645+0.125) {$\scriptstyle t={-4/3\,(1+w_{b}^{3/2})}$};
    \end{scope}
\end{tikzpicture}
\caption{The convergence disks of the expansions around
$t=0,-4/3,-4/3\,(1+w_{b}^{3/2})$ with $w_{b}=e^{\, i \varphi_{b}}$
and $\varphi_{b}=\frac{2}{3}\arctan(-7/8)$ are shown above. The
expansion around $t=-\frac{4}{3} (s=0)$ corresponds to the expansion
in Lemma \ref{lema1} and has a converge disk of radius
$\rho=\frac{4}{3}$. The expansion around $t=0$ has a decreasing
radius for its convergence disk as $\varphi \rightarrow
\frac{2\pi}{3}$ due to the approach of the branch point at
$t=-4/3\,(1+w^{3/2})$. When $\frac{2}{3}\arctan(-7/8)\le |\varphi| <
\frac{2\pi}{3}$, the series expansion is done around the branch
point instead of around $t=0$ as described in Lemma \ref{lema1d}.}
\label{fig:covergencedisks}
\end{figure}

\begin{lemma}
The power series of $\alpha=\alpha(t)$ around $t=-4/3\,(1+w^{3/2})$
for equation \eqref{eq-chvar2} and $\arg z > 0$ is given by equation
\eqref{eq-seriescomplex3a} with $t+4/3\,(1+w^{3/2})$ replacing $t$
and with $\alpha_{0}^{+}=w^{1/2}$ and $\alpha_{0}^{-}=-2\,w^{1/2}$.
As in Lemma \ref{lema1c}, the superscripts $+$ must be replaced by
$-$ if $\arg z <0$, so that the series analogous to equation
\eqref{eq-seriescomplex3a} is now for $\alpha^{-}$ and the binomial
inside is $(\alpha_{0}^{-}-\alpha_{0}^{+})$ instead of
$(\alpha_{0}^{+}-\alpha_{0}^{-})$. The radius of convergence is
given by $\rho=\frac{4}{3}$.
\label{lema1d}

\end{lemma}
\begin{proof}[\bfseries Proof]
The demonstration is analogous to the one of Lemma \ref{lema1c}. The
radius of convergence of the resulting series around
$t=-4/3\,(1+w^{3/2})$ is $\rho=\frac{4}{3}$, as can be easily seen
in Figure \ref{fig:sheets2}a). A more detailed view of the
convergence disk is given in Figure \ref{fig:covergencedisks}. An
expansion around $t=0$ has a decreasing radius as $-4/3\,(1+w^{3/2})
\rightarrow 0$ and therefore it is preferable to expand around
$t=-4/3(1+w^{3/2})$ for $\varphi$ in $w=e^{\, i \varphi}$ such that
$\varphi_{b}\equiv\frac{2}{3}\arctan(-7/8)\le |\varphi| <
\frac{2\pi}{3}$. In effect, when $\varphi=\varphi_{b}$, a
circumference of radius $1/2\,\rho$ centred at $t=-4/3(1+w^{3/2})$
crosses the real axis of the complex plane for $t$ at $t=-2/3$ and
$t=1/3$. If we now consider the angular range given by
$\varphi_{b}\equiv\frac{2}{3}\arctan(-7/8)\le |\varphi| <
\frac{2\pi}{3}$, it can be readily seen in Figure
\ref{fig:covergencedisks} that the segment $t\in[-1/3,1/3]$ makes
use of half the convergence radius at most.

\end{proof}


With the preceding four lemmas we have computed series expansions
that can be used to integrate the path sections
$\mathcal{L}_{\alpha_{0}^{-}}$, $\mathcal{L}_{s}$ and
$\mathcal{L}_{\alpha_{0}^{+}}$ in Figure
\ref{fig:convergenceregions}. For the purpose of integrating along
the section of the path that connects the convergency disks of the
series in Lemmas \ref{lema1c} and \ref{lema1d} with $s=\infty$ in
\eqref{eq-chvar1}, a new variable $\hat{s}=1/s^{2/3}$ is introduced,

\begin{equation}
w\alpha-\frac{\alpha^{3}}{3}=-\frac{2}{3}w^{3/2}-\frac{1}{\hat{s}^{3}}.
\label{eq-alphau}
\end{equation}
The resulting complex algebraic curve is analyzed in the following
Remark, and the series expansions for $\alpha=\hat{\alpha}(\hat{s})$
around $\hat{s}$ is computed in the subsequent Lemma. 

\begin{figure}

\begin{tabular}{c c}
\resizebox{6cm}{!}{\includegraphics{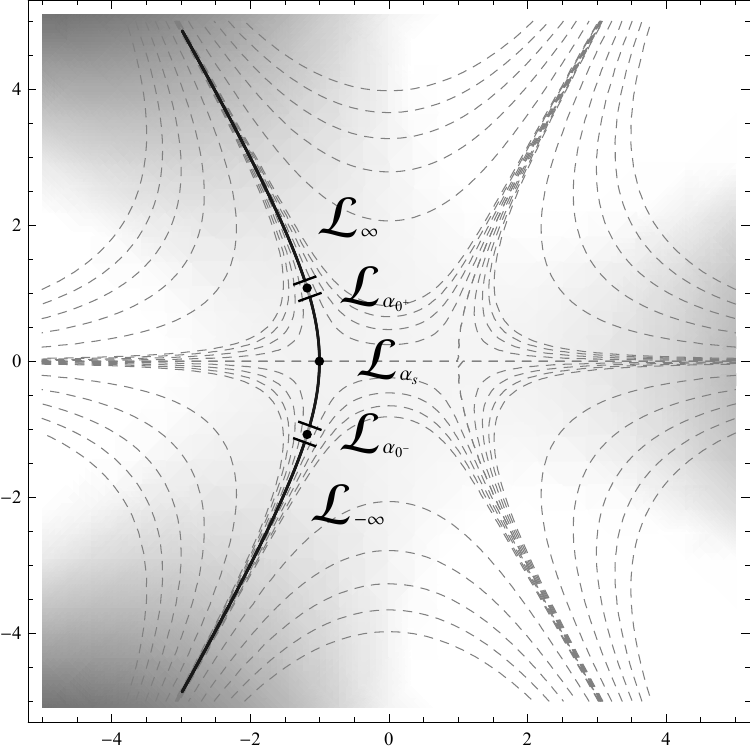}} &
\resizebox{6cm}{!}{\includegraphics{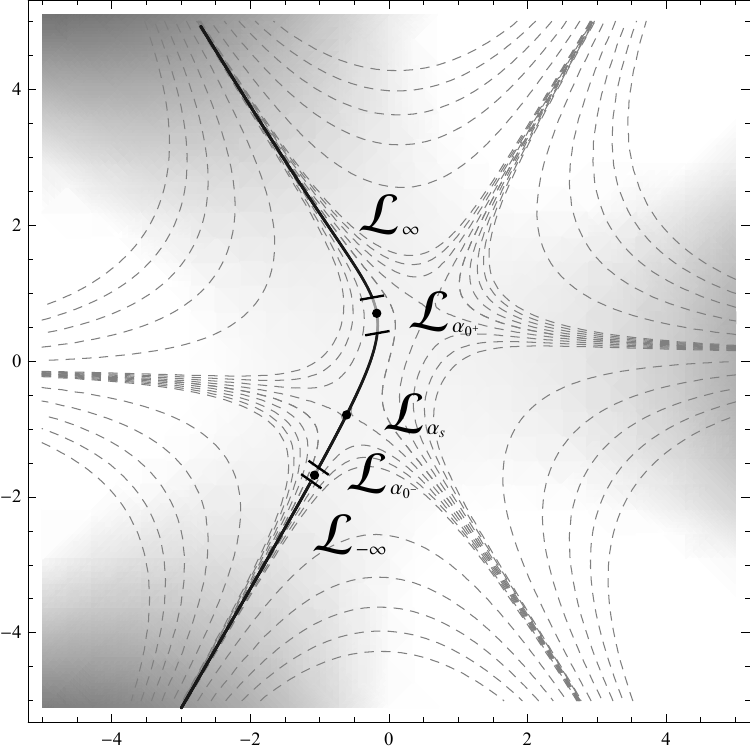}}
\vspace{-7pt} \\
{\scriptsize(a)} ${\scriptstyle \arg z = 0}$ & {\scriptsize(b)} ${\scriptstyle \arg z = 7\pi/12}$ \\
\resizebox{6cm}{!}{\includegraphics{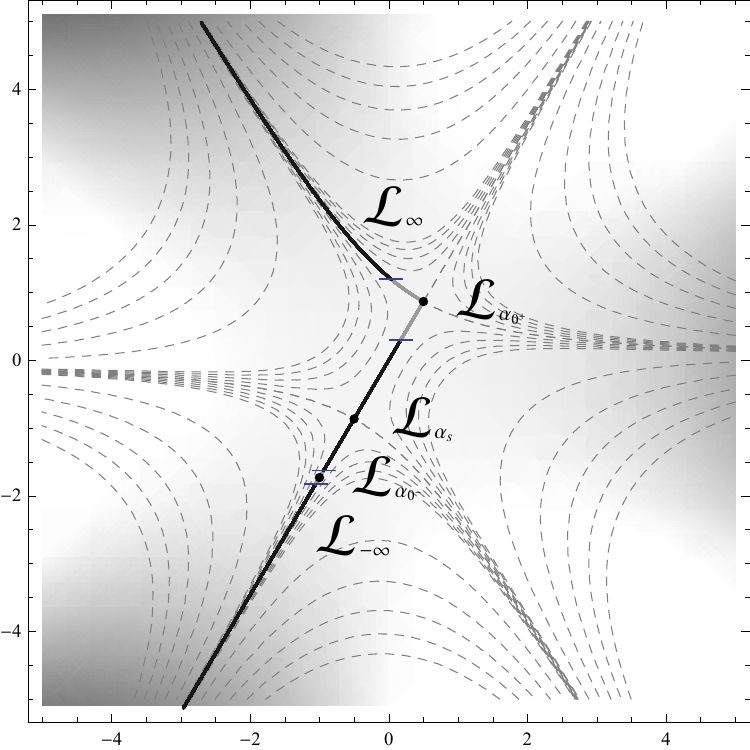}} &
\resizebox{6cm}{!}{\includegraphics{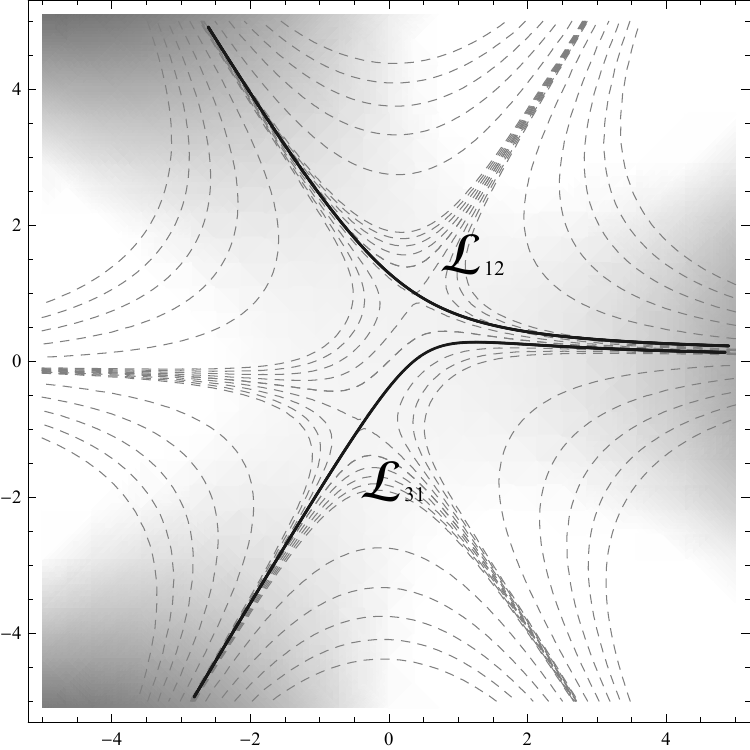}}
\vspace{-7pt} \\
{\scriptsize(c)} ${\scriptstyle \arg z = 2\pi/3}$ & {\scriptsize(d)}
${\scriptstyle \arg z = 3\pi/4}$
\end{tabular}
\caption{Steepest-descent integration paths for different phases of
the $z$ variable. The background colour of the complex plane
corresponds to the value of the real part of the argument of the
exponential as given in equation \eqref{eq-taylorscomplex0},
$f_{r}$: the white-to-black scale maps to higher-to-lower values of
such argument. The points around which the functions $\alpha(s)$ and
$\alpha(t)$ are expanded appear with a black dot. Five sections are
distinguished in each integration path when $|\arg z| \le 2\pi/3$.
If $|\arg z| > 2\pi/3$, the steepest descent paths are homotopic to
$\mathcal{L}_{31}$ and $\mathcal{L}_{12}$ in Figure
\ref{fig:paths}.}
\end{figure}
\setcounter{figure}{5}
\begin{figure}
\begin{tabular}{c c}
\resizebox{6cm}{!}{\includegraphics{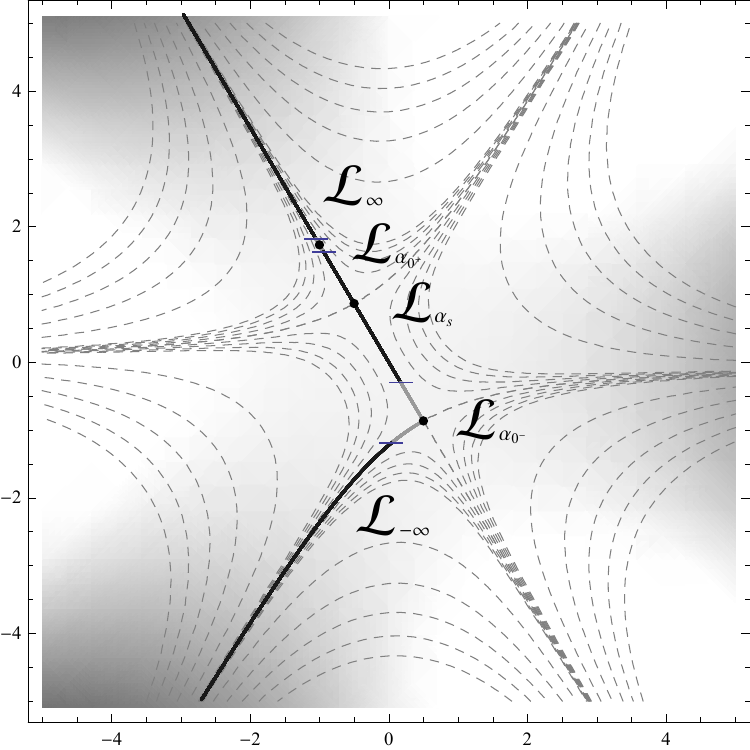}} &
\resizebox{6cm}{!}{\includegraphics{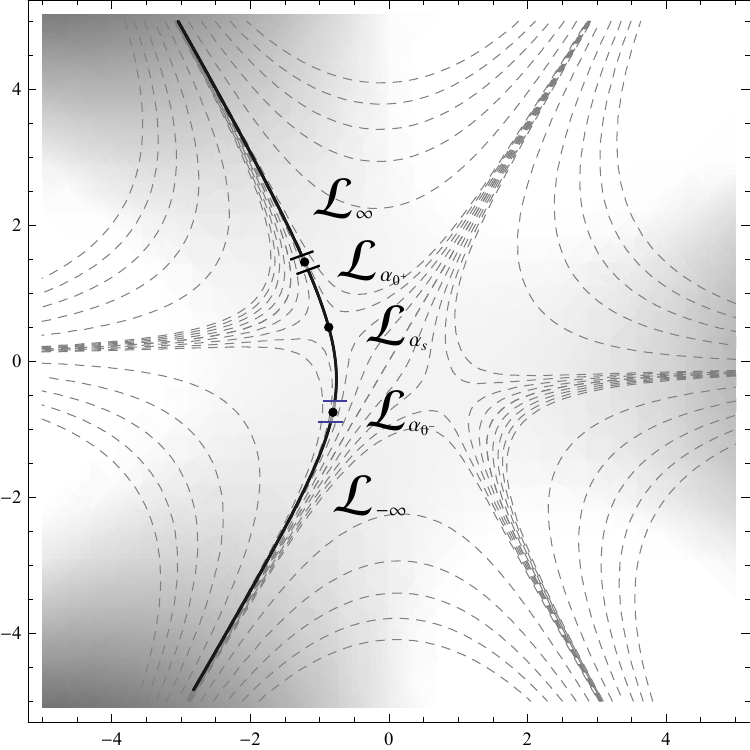}}
\vspace{-7pt} \\
{\scriptsize(e)} ${\scriptstyle \arg z = -2\pi/3}$ &
{\scriptsize(f)} ${\scriptstyle \arg z = -\pi/3}$
\\
\end{tabular}
\caption{(Cont.) Steepest-descent integration paths for different
phases of the $z$ variable. The background colour of the complex
plane corresponds to the value of the real part of the argument of
the exponential as given in equation \eqref{eq-taylorscomplex0},
$f_{r}$: the white-to-black scale maps to higher-to-lower values of
such argument. The points around which the functions $\alpha(s)$ and
$\alpha(t)$ are expanded appear with a black dot. Five sections are
distinguished in each integration path when $|\arg z| \le 2\pi/3$.
If $|\arg z| > 2\pi/3$, the steepest descent paths are homotopic to
$\mathcal{L}_{31}$ and $\mathcal{L}_{12}$ in Figure
\ref{fig:paths}.}
\label{fig:intpaths}
\end{figure}
\begin{remark}
{\itshape The three solutions of \eqref{eq-alphau} in terms of
$\hat{s}$ are given by
\begin{align}
\hat{\alpha}_{1}(\hat{s})&=\frac{\chi(\hat{s}
\sqrt{w})}{\sqrt[\leftroot{-1}\uproot{2}\scriptstyle 3]{2}\,\hat{s}}
+\frac{\sqrt[\leftroot{-1}\uproot{2}\scriptstyle
3]{2}\,\hat{s}\,w}{\chi(\hat{s} \sqrt{w})}\nonumber
\\
\hat{\alpha}_{2}(\hat{s})&=\frac{e^{i 2 \pi/3}\,\chi(\hat{s}
\sqrt{w})} {2^{1/3}\,\hat{s}}+\frac{2^{1/3}\,\,e^{-i 2
\pi/3}\,\hat{s}\,w}{\chi(\hat{s} \sqrt{w})} \nonumber
\\
\hat{\alpha}_{3}(\hat{s})&=\frac{e^{-i 2 \pi/3}\,\chi(\hat{s}
\sqrt{w})} {2^{1/3}\,\hat{s}}+\frac{2^{1/3}\,\,e^{i 2
\pi/3}\,\hat{s}\,w}{\chi(\hat{s} \sqrt{w})} \nonumber
\\
&\chi(t)=[3+2\,t^{3}+\sqrt{3}\, \eta(t)]^{1/3}\nonumber
\\
&\eta(t)= \sqrt{3+4\,t^{3}}.
\label{eq-alphahats}
\end{align}
Solutions $\hat{\alpha}_{2}(\hat{s})$ y $\hat{\alpha}_{3}(\hat{s})$
define the integration paths $\mathcal{L}_{-\infty}$ and
$\mathcal{L}_{\infty}$ for the case of $|\arg w|=|\arg z|\le
2\pi/3$, respectively.

Contrary to the case seen in the previous Lemmas, Taylor's theorem
is preferred over Lagrange inversion theorem for the computation of
\eqref{eq-phisai} with \eqref{eq-alphahats}. To perform the required
derivatives, functional composition must be used throughout the
whole derivation. In particular, Fa\`{a} di Bruno's formula
expressed in terms of Bell polynomials $B_{n,k}$ is recursively used
to deal with the functional form of the solutions in
\eqref{eq-alphahats}.

Taylor's theorem is to be applied separately to each addend of the
expressions of $\hat{\alpha}_{i}(\hat{s}),\, i=1,2,3$, and therefore
the branch points that limit the radii of convergence are not the
ones corresponding to the complex algebraic curve in
\eqref{eq-alphau} but to the one of which they are solutions. It is
straightforward to see that such a curve is
\begin{equation}
\hat{s}^{3}\,z^{2}-\left(2\,w^{3/2}\,\hat{s}^{3}-3\right)z+\hat{s}^{3}\,w^{3}=0.
\label{eq-cal3}
\end{equation}
Its branch points are found in $\hat{s}=\zeta, e^{i 2 \pi/3}\,\zeta$
and $e^{-i 2 \pi/3}\,\zeta$ with
$\zeta=\left(\frac{3}{4}\right)^{1/3}\,w^{-1/2}$. There are only two
sheets in $z=z(\hat{s})$ for \eqref{eq-cal3} and all the branch
points connect them. Therefore, the power expansions that will
result from the application of Taylor's theorem around $\hat{s}=0$
will have radii of convergence given by
$\rho=\left(\dfrac{3}{4}\right)^{1/3}$.}
\label{rem2}
\end{remark}

\begin{lemma}
The $n$-th derivative of $\chi$ at $t=0$ is given for $n \neq 0$ by
\begin{equation}
\chi^{(n)}(0)=\left\{\begin{array}{l
l}6^{\frac{1}{3}}\sum_{k=1}^{n/3}\sum_{i=0}^{k}\sum_{l=0}^{k-i}(-1)^{k-i-l}
\frac{2^{i-k}n!}{3^{i}\,i!\left(\frac{n}{3}-i\right)! l!(k-i-l)!}&
\\
\qquad\qquad \times \left(\frac{1}{3}\right)_{k}
h_{\frac{n}{3}-i}\left(\frac{l}{2}\right) & \mbox{if } n \equiv
0\!\! \pmod 3
\\

\\
0 & \mbox{ otherwise}
\end{array}\right.
\label{eq-chinthder0}
\end{equation}
where
\begin{equation}
h_{p}(x)=\left(\frac{4}{3}\right)^{p}(x)_{p}
\end{equation}
and $(x)_{p}$ is the falling factorial of $x$ of order $p$.
\label{lema2}
\end{lemma}
\begin{proof}[\bfseries Proof]
The following expression is obtained for the $n$-th derivative of
$\chi$ at $t=0$ after using Fa\`{a} di Bruno's formula expressed in
terms of Bell polynomials $B_{n,k}$,
\begin{align}
\chi^{(n)}(0)&=6^{\frac{1}{3}}\sum_{k=1}^{n}\left(\frac{1}{3}\right)_{k}
\frac{1}{6^{k}}\nonumber
\\
& \times
B_{n,k}(0,0,12+\sqrt{3}\,\eta^{(3)}(0),0,0,\sqrt{3}\,\eta^{(6)}(0),0,\ldots)
\end{align}
with
\begin{equation}
\eta^{(3 m)}(0)=
\sqrt{3}\left(\dfrac{1}{2}\right)_{m}\dfrac{2^{2m}\,(3
m)_{2m}}{3^{m}}
\end{equation}
and $(x)_{n}=x(x-1)\cdots(x-n+1)$. By using Bell's polynomial
definition, this expression can be simplified to
\begin{multline}
B_{n,k}(0,0,\sqrt{3}\eta^{(3)},0,0,\sqrt{3}\eta^{(6)},0,0,\sqrt{3}\eta^{(9)},\ldots)
=3^{k}\,(n)_{2n/3}
\\
\times\left\{\begin{array}{l l}
B_{n/3,k}(h_{1}(\frac{1}{2}),h_{2}(\frac{1}{2}),h_{3}(\frac{1}{2}),\ldots,h_{n/3-k+1}(\frac{1}{2}))&
\mbox{if } n \equiv 0\!\! \pmod 3
\\
0 & \mbox{ otherwise}
\end{array}\right.
\label{eq-Bn1}
\end{multline}
with
\begin{equation}
h_{p}(x)=\left(\frac{4}{3}\right)^{p}(x)_{p}
\end{equation}
forming a binomial sequence \cite{yang:2008}. 
Equation \eqref{eq-chinthder0} is obtained, after some manipulation,
by applying the following two properties to \eqref{eq-Bn1}:
\begin{itemize}
\item For a binomial sequence $\{\varphi_{n}(x)\}$ and all integers
$m,k\geq 0$, we have \cite{yang:2008}
\begin{equation}
B_{m,k}(\varphi_{1}(x),\varphi_{2}(x),\varphi_{3}(x),\ldots)=\sum_{j=0}^{k}\frac{(-1)^{k-j}}{j!\,(k-j)!}\varphi_{m}(j\frac{1}{2})
\end{equation}
\item Sums in the variables of Bell's polynomials can be written as
\cite{comtet:74} 
\begin{multline}
B_{m,k}(y_{1}+y_{1}^{\prime},y_{2}+y_{2}^{\prime},y_{3}+y_{3}^{\prime},\ldots)
\\
=\sum_{\substack{i\leq m \\ j\leq k}}\binom{m}{i}
B_{i,j}(y_{1},y_{2},y_{3},\ldots)
B_{m-i,k-j}(y_{1}^{\prime},y_{2}^{\prime},y_{3}^{\prime},\ldots)
\end{multline}
\end{itemize}
\end{proof}

\begin{corollary}
The functions $\alpha_{n}(\hat{s})$ $(n=2,3)$ can be written as the
following power series
\begin{align}
\hat{\alpha}_{2}&=-\frac{1-i\sqrt{3}}{2^{4/3}}
\sum_{m=0}^{\infty}\frac{a_{3 m}}{(3 m)!} w^{3 m/2} \hat{s}^{3
m-1}-\frac{1+i\sqrt{3}}{2^{2/3}}w\sum_{m=0}^{\infty}\frac{b_{3
m}}{(3 m)!} w^{3 m/2} \hat{s}^{3 m+1}\nonumber
\\
\hat{\alpha}_{3}&=-\frac{1+i\sqrt{3}}{2^{4/3}}
\sum_{m=0}^{\infty}\frac{a_{3 m}}{(3 m)!} w^{3 m/2} \hat{s}^{3
m-1}-\frac{1-i\sqrt{3}}{2^{2/3}}w\sum_{m=0}^{\infty}\frac{b_{3
m}}{(3 m)!} w^{3 m/2} \hat{s}^{3 m+1}
\label{eq-corollary1}
\end{align}
with
\begin{align}
a_{3 m}&=\chi^{(3 m)}(0)\nonumber
\\
b_{3
m}&=6^{-\frac{1}{3}}\left\{\delta_{m\,0}+\sum_{k=1}^{m}\sum_{i=0}^{k}\sum_{l=0}^{k-i}(-1)^{k-i-l}
\frac{2^{i-k}(3
m)!}{3^{i}\,i!\left(m-i\right)!l!(k-i-l)!}\right.\nonumber
\\
&\quad\quad \times \left.\left(\!-\frac{1}{3}\right)_{k}
h_{m-i}\left(\frac{l}{2}\right)\right\}
\label{eq-a3mb3m}
\end{align}
which are uniformly convergent for $|\hat{s}|<\rho=(3/4)^{1/3}$.
\label{cor-coeff}
\end{corollary}
\begin{proof}[\bfseries Proof]
These series result from the use of Fa\`{a} di Bruno's formula,
Lemma \ref{lema2} and the Remark in \eqref{eq-alphahats}. They are
uniformly convergent for any compact set inside the convergence
disk.
\end{proof}

After applying Lemmas \ref{lema1} to \ref{lema1d}, the entire
steepest descent path can be segmented in five pieces as shown in
Figure \ref{fig:convergenceregions}. These segments are centred at
$s=0$, $\pm 2/\sqrt{3}$, and $\pm \infty$ and their extremes are
given by $s=-\infty,-\sqrt{5/3},-1,0,1,\sqrt{5/3},\infty$. The
result of integrating \eqref{eq-aiz1} for $|\mbox{arg } z| \leq
2\pi/3$ is given by the following theorem.

\begin{theorem}
\label{th-airyseries}
The Airy function $\mbox{\textup{Ai}}(z)$ for $z\in\mathds{C}$ with
$|\mbox{arg }z|\leq 2\pi/3$, is given by the sum of the following
convergent series expansions
\begin{equation}
\mbox{\textup{Ai}}(z)=I_{\mathcal{L}_{-\infty}}(z)+
I_{\mathcal{L}_{\alpha_{0}^{-}}}(z)+I_{\mathcal{L}_{s}}(z)
+I_{\mathcal{L}_{\alpha_{0}^{+}}}(z)+I_{\mathcal{L}_{\infty}}(z)
\label{eq-aiIexp}
\end{equation}
with
\begin{align}
I_{\mathcal{L}_{\alpha_{s}}}(z)&=\frac{1}{2\pi
z^{1/4}}e^{-2/3\,z^{3/2}}\sum_{n=0}^{\infty}(-1)^{n}
\frac{\Gamma(3n+\frac{1}{2})}{3^{2n}(2n)!}
\frac{\gamma(n+\frac{1}{2},s_{0}^{2}\,|z|^{3/2})}{\Gamma(n+\frac{1}{2})}
z^{-3/2 n}\nonumber
\\
I_{\mathcal{L}_{-\infty}}(z)+&
I_{\mathcal{L}_{\infty}}(z)=-\frac{\sqrt{3}}{2^{2/3}\pi}e^{-2/3\,z^{3/2}}
\nonumber
\\
&\times \sum_{m=0}^{\infty}\left\{\frac{1}{2^{2/3}}\frac{a_{3m}}{(3
m)!} \big(m-\frac{1}{3}\big)\, \Gamma\left(\frac{1}{3}-m,s_{1}^2\,
|z|^{3/2}\right) z^{3/2 m}\right.\nonumber
\\
&\qquad\left.-\frac{b_{3m}}{(3 m)!}\big(m+\frac{1}{3}\big)\,
\Gamma\left(-\frac{1}{3}-m,s_{1}^2\, |z|^{3/2}\right)\,z^{3/2
m+1}\right\}
\label{eq-th1}
\end{align}
and
\begin{align}
I_{\mathcal{L}_{\alpha_{0}^{\pm}}}(z)&=\pm\frac{|z|^{1/2}}{2\pi
i}e^{-2/3\,z^{3/2}}e^{-4/3\,|z|^{3/2}}\sum_{n=0}^{\infty}\frac{(-1)^{n-1}}{(n-1)!}
3^{n} |z|^{-3 n/2} \nonumber
\\
&\times \left[\Gamma(n,\mp t_{0}|z|^{3/2})-\Gamma(n,\pm
t_{0}|z|^{3/2})\right](\alpha_{0}^{\pm}-\alpha_{0}^{\mp})^{-2 n+1}
(\alpha_{0}^{\pm}-\alpha_{0})^{-n}\nonumber
\\
&\times\sum_{k=0}^{n-1}\binom{n-1}{k}\frac{\Gamma(2
n-1-k)}{\Gamma(n)}\frac{\Gamma(n+k)}{\Gamma(n)}\left(\frac{\alpha_{0}^{\pm}-\alpha_{0}^{\mp}}{\alpha_{0}^{\pm}-\alpha_{0}}\right)^{k}
\label{eq-th1b}
\end{align}
for $w\neq e^{\, i 2\pi/3}$ and with $\alpha_{0}^{\pm}$ as given in
Lemma \ref{lema1b}, or, with $s=\sign\{\arg z\}$,
\begin{align}
I_{\mathcal{L}_{\alpha_{0}^{+s}}}(z)&=\frac{|z|^{1/2}}{4\pi
i}e^{-2/3\,z^{3/2}}e^{-4/3\,|z|^{3/2}}\sum_{n=0}^{\infty}\frac{(-1)^{2
n-1}}{(n-1)!} 3^{n/2} \frac{\Gamma(\frac{3
n}{2}-1)}{\Gamma(\frac{n}{2})}\nonumber
\\
&\times (\alpha_{0}^{+s}-\alpha_{0}^{-s})^{-3 n/2+1} |z|^{-3 n/4}
\left[\Gamma(\frac{n}{2},\unaryminus s\,
t_{0}|z|^{3/2})\mid_{b}-\Gamma(\frac{n}{2},s\,t_{0}|z|^{3/2})\right]\nonumber
\\
I_{\mathcal{L}_{\alpha_{0}^{-s}}}(z)&=\frac{|z|^{1/2}}{2\pi
i}e^{-2/3\,z^{3/2}}e^{-4/3\,|z|^{3/2}}\sum_{n=0}^{\infty}\frac{(-1)^{n-1}}{(n-1)!}
3^{n} \frac{\Gamma(3 n-1)}{\Gamma(2 n)}\nonumber
\\
&\times (\alpha_{0}^{-s}-\alpha_{0}^{+s})^{-3 n+1} |z|^{-3 n/2}
\left[\Gamma(n,s\,t_{0}|z|^{3/2})-\Gamma(n,\unaryminus s\,
t_{0}|z|^{3/2})\right]
\label{eq-th1c}
\end{align}
for $w= e^{\, i 2\pi/3}$ and with $\alpha_{0}^{\pm}$ as given in
Lemma \ref{lema1c}, where $\mid_{b}$ is the branch in the case of
$\Gamma(\frac{n}{2},\unaryminus t_{0}|z|^{3/2})$, which implies to
take the complex conjugate for the case of $\arg z < 0$. The values
of $s_{0}$, $s_{1}$ and $t_{0}$ are given by
\begin{align}
s_{0}&=1\nonumber
\\
s_{1}&=\sqrt{\frac{5}{3}}\nonumber
\\%
t_{0}&=\frac{1}{3}.
\label{eq-integrationrange}
\end{align}
The series for $I_{\mathcal{L}_{\alpha_{0}^{\pm}}}(z)$ in
\eqref{eq-th1b} is slowly convergent when $|\arg z| <
\frac{2}{3}\arctan(-7/8)$ as a consequence of the observations in
Lemma \ref{lema1d} and can be replaced by
\begin{align}
I_{\mathcal{L}_{\alpha_{0}^{+s}}}(z)&=\frac{|z|^{1/2}}{4\pi
i}e^{2/3\,z^{3/2}}\sum_{n=0}^{\infty}\frac{(-1)^{2 n-1}}{(n-1)!}
3^{n/2} \frac{\Gamma(\frac{3
n}{2}-1)}{\Gamma(\frac{n}{2})}(\alpha_{0}^{+s}-\alpha_{0}^{-s})^{-3
n/2+1} |z|^{-3 n/4}\nonumber
\\
&\times \left[\Gamma(\frac{n}{2},\unaryminus s
(t_{0}+\hat{t})|z|^{3/2})-\Gamma(\frac{n}{2},s(t_{0}-\hat{t})|z|^{3/2})\right]
\label{eq-th1d}
\end{align}
with $\alpha_{0}^{\pm}$ and $\hat{t}=-4/3\,(1+w^{3/2})$ as given in
Lemma \ref{lema1d}.

Functions $\gamma(\nu,x)$ and $\Gamma(\nu,x)$ are the lower and
upper incomplete gamma functions, respectively.

For the case where $|\mbox{\textup{arg} }z|> 2\pi/3$, the following
property is used,
\begin{equation}
\mbox{\textup{Ai}}(z)=-e^{i 2\pi/3}\, \mbox{\textup{Ai}}(e^{i
2\pi/3}\, z)-e^{-i 2\pi/3}\, \mbox{\textup{Ai}}(e^{-i 2\pi/3}\, z).
\label{eq-Aipw}
\end{equation}
Similarly, the Airy function of the second kind,
$\mbox{\textup{Bi}}(z)$ can be computed from the series expressions
for $\mbox{\textup{Ai}}(z)$, with the use of the property
\eqref{eq-biai}
\begin{equation}
\mbox{\textup{Bi}}(z)=e^{i\,\pi/6}\mbox{\textup{Ai}}(z\,e^{i\,2\pi/3})+e^{-i\,\pi/6}\mbox{\textup{Ai}}(z\,e^{-i\,2\pi/3}).
\label{eq-biai2}
\end{equation}
\end{theorem}
\begin{proof}[\bfseries Proof]
The integration in \eqref{eq-ai2} can be done through the five
sections defined by the expansions around points $s=0, \pm\infty$
and $t=0$ or $t=4/3\,(1+w^{3/2})$, as described in Lemmas
\ref{lema1} to \ref{lema1d} and Corollary \ref{cor-coeff}. The
choice of the limiting points between sections is made to guarantee
uniform and fast convergence of the series under integration and
also to maintain symmetric integration bounds. In particular, we can
see the case of integrating through the segments containing
$\alpha_{0}^{+}$ y $\alpha_{0}^{-}$, as shown in Figure
\ref{fig:convergenceregions} (case of $\varphi=\arg z > 0$): the
expansion around $\alpha_{0}^{+}$ has a radius of convergence,
$\rho=m$, that is more restrictive than the one around
$\alpha_{0}^{-}$, which is $\rho=\frac{4}{3}$. As $\varphi$
approaches $\frac{2}{3} \pi$, the radius becomes zero. At
$\varphi=\frac{2}{3}\arctan(-\frac{7}{8})$, we have $m=\frac{2}{3}$.
Therefore, the integration in $t\in [-\frac{1}{3},\frac{1}{3}]$ uses
only half the convergence radius. For values
$\frac{2}{3}\arctan(-\frac{7}{8}) < \varphi < \frac{2}{3} \pi$, we
use the expansion around $t=-\frac{4}{3}\,(1+w^{3/2})$, which has a
convergence disk with $\rho=\frac{4}{3}$ independently of $\varphi$.
However, this expansion does not properly accommodate the
integration segment, $t\in [-\frac{1}{3},\frac{1}{3}]$, for
$\varphi<\frac{2}{3}\arctan(-\frac{7}{8})$: either it brings this
segment too close to the convergence disk boundary or out of it. On
the other hand, the expansion around $\alpha_{0}^{-}$ can include in
similar conditions a bigger segment than $t\in
[-\frac{1}{3},\frac{1}{3}]$, such as $t\in
[-\frac{2}{3},\frac{2}{3}]$. The use of such a segment would produce
an asymmetry in the integration bounds for the Gamma function that
are to be avoided when possible, and, therefore, $t\in
[-\frac{1}{3},\frac{1}{3}]$ is again chosen for integrating the
corresponding expansion, as shown in Figure
\ref{fig:convergenceregions}. 

Therefore, starting with the application of Lemma \ref{lema1} in
equation \eqref{eq-aiz1} and with the integration limits given by
$\unaryminus s_{0}$ and $s_{0}$, we obtain
\begin{align}
I_{\mathcal{L}_{s}}(z)&=\frac{1}{2\pi\,i}\int_{\mathcal{L}_{s}}du\,e^{x
u-1/3 u^{3}}=\frac{|z|^{1/2}}{2\pi
i}\int_{\mathcal{L}_{s}}d\alpha\,e^{|z|^{3/2}( w\alpha-1/3
\alpha^{3})}\nonumber
\\
&=\frac{1}{2\pi z^{1/4}}e^{-2/3\,z^{3/2}}\sum_{n=0}^{\infty}(-1)^{n}
\frac{\Gamma(3n+\frac{1}{2})}{3^{2n}(2n)!}
\frac{\gamma(n+\frac{1}{2},s_{0}^{2}|z|^{3/2})}{\Gamma(n+\frac{1}{2})}
z^{-3/2 n}.
\label{eq-I1c}
\end{align}
Function $\gamma(s,t)$ is the lower incomplete gamma function. In
effect, this comes from
\begin{align}
\int_{-a}^{a}s^{2n+1}\exp[-\sigma\,s^{2}]\,ds&= 0\nonumber
\\
\int_{-a}^{a}s^{2n}\exp[-\sigma\,s^{2}]\,ds&=
\frac{\gamma(n+\frac{1}{2},a^{2} \sigma)}{\sigma^{n+1/2}}
\label{eq-gammaintegrals3}
\end{align}
where $n\in\mathds{N}$, $\sigma\in\mathds{R}^{+}$ and $a>0$.

As for the integrals in $\mathcal{L}_{-\infty}$ and
$\mathcal{L}_{\infty}$, they are computed here through the change of
variable given in \eqref{eq-alphau}. This makes it possible to
compute an expansion of the solutions of $\alpha=\alpha(\hat{s})$
for $|\hat{s}|<(4/3)^{1/3}$ as corresponding to
$s\in(-\infty,-2/\sqrt{3})$ and $s\in(2/\sqrt{3},\infty)$. As stated
in Remark \ref{rem2}, the solutions $\alpha_{2}(\hat{s})$ and
$\alpha_{3}(\hat{s})$ correspond to the curves
$\mathcal{L}_{-\infty}$ and $\mathcal{L}_{\infty}$, respectively.
Corollary \ref{cor-coeff} produces
\begin{align}
\left.\Phi(s)\right|_{\mathcal{L}_{-\infty}}&=-\frac{1+i\sqrt{3}}{2^{4/3}}\sum_{m=0}^{\infty}
\frac{a_{3m}}{(3m)!} w^{3 m/2}
(2m-\frac{2}{3})\,s^{-2m-\frac{1}{3}}\nonumber
\\
&-\frac{1-i\sqrt{3}}{2^{2/3}} w \sum_{m=0}^{\infty}
\frac{b_{3m}}{(3m)!} w^{3 m/2}
(2m+\frac{2}{3})\,s^{-2m-\frac{5}{3}}\nonumber
\\
\left.\Phi(s)\right|_{\mathcal{L}_{\infty}}&=\frac{1-i\sqrt{3}}{2^{4/3}}\sum_{m=0}^{\infty}
\frac{a_{3m}}{(3m)!} w^{3 m/2}
(2m-\frac{2}{3})\,s^{-2m-\frac{1}{3}}\nonumber
\\
&+\frac{1+i\sqrt{3}}{2^{2/3}} w \sum_{m=0}^{\infty}
\frac{b_{3m}}{(3m)!} w^{3 m/2}
(2m+\frac{2}{3})\,s^{-2m-\frac{5}{3}}.
\label{eq-I2c}
\end{align}
Therefore, the contribution to the Airy's integral by
$\mathcal{L}_{-\infty}$ and $\mathcal{L}_{\infty}$ is
\begin{align}
I_{\mathcal{L}_{-\infty}}(z)+I_{\mathcal{L}_{\infty}}(z)&=-\frac{\sqrt{3}}{2^{2/3}\pi}e^{-2/3\,z^{3/2}}
\nonumber
\\
&\times \sum_{m=0}^{\infty}\left\{\frac{1}{2^{2/3}}\frac{a_{3m}}{(3
n)!} \big(m-\frac{1}{3}\big)\,
\Gamma\left(\frac{1}{3}-m,\frac{4}{3}|z|^{3/2}\right) z^{3/2
m}\right.\nonumber
\\
&\qquad\left.-\frac{b_{3m}}{(3 m)!}\big(m+\frac{1}{3}\big)\,
\Gamma\left(-\frac{1}{3}-m,\frac{4}{3}|z|^{3/2}\right)\,z^{3/2
m+1}\right\}
\label{eq-I3c}
\end{align}
since
\begin{equation}
\int_{a}^{\infty}s^{-2n-k/3}\exp[-\sigma\,s^{2}]\,ds=
\frac{\Gamma(\frac{1}{2}-\frac{k}{6}-n,a^{2}
\sigma)}{2\,\sigma^{\frac{1}{2}-\frac{k}{6}-n}}
\label{eq-gammaintegrals4}
\end{equation}
with $n,k\in\mathds{N}$, $\sigma\in\mathds{R}^{+}$ and $a>0$.

In a similar manner, Lemmas \ref{lema1b}, \ref{lema1c} and
\ref{lema1d} produce equations \eqref{eq-th1b}, \eqref{eq-th1c} and
\eqref{eq-th1d} by substituting the corresponding $\alpha$-series
into equation \eqref{eq-phisai} and by using
\begin{align}
\int_{-t_{0}}^{t_{0}}
t^{n-1}\exp[-\sigma\,t]\,dt&=\sigma^{-n}\left[\Gamma(n,\unaryminus
t_{0}\sigma)-\Gamma(n,t_{0}\sigma)\right]\nonumber
\\
\int_{-t_{0}}^{t_{0}}
t^{n/2-1}\exp[-\sigma\,t]\,dt&=\sigma^{-n/2}\left[\Gamma(\frac{n}{2},\unaryminus
t_{0}\sigma)\mid_{b}-\Gamma(\frac{n}{2},t_{0}\sigma)\right]\nonumber
\\
\int_{-t_{0}}^{t_{0}}
\left(t-\hat{t}\right)^{n/2-1}\exp[-\sigma\,t]\,dt&=\sigma^{-n/2}\,e^{-\hat{t}\sigma}\left[\Gamma(\frac{n}{2},\unaryminus
(t_{0}+\hat{t})\sigma)-\Gamma(\frac{n}{2},(t_{0}-\hat{t})\sigma)\right]
\end{align}
with $\mid_{b}$ being the branch of $t^{n/2}$, which implies to take
the complex conjugate for the case of $\arg z < 0$ as results from
Lemma \ref{lema1c}.

The five-section decomposition is valid for $|\mbox{arg }z|\leq
2\pi/3$ as has just been described, and it is illustrated in Figures
\ref{fig:intpaths}(a) to \ref{fig:intpaths}(f), except in Figure
\ref{fig:intpaths}(d). The latter shows the case of $|\mbox{arg }z|
> 2\pi/3$, for which the integration path splits in two separate
paths. Integration through the two paths would require a different
set of integrals. However, the well known property of the Airy
function given by equation \eqref{eq-Aipw} allows to identify the
mapping between the lower path in Figure \ref{fig:intpaths}(d) with
the first term in \eqref{eq-Aipw} and the upper path with its second
term\footnote{This can be checked by realizing that the
$\alpha$-solutions of the steepest descent paths for the saddle
point $\alpha_{s}=w^{1/2}$ are given by analogous solutions to
\eqref{eq-alphahats} but with
\begin{align}
\bar{\chi}(t)&=[3-2\,t^{3}+\sqrt{3}\, \bar{\eta}(t)]^{1/3}\nonumber
\\
\bar{\eta}(t)&= \sqrt{3-4\,t^{3}}
\end{align}
replacing $\chi(t)$ and $\eta(t)$.}. A similar correspondence is
found for the case of $\mbox{Bi}(z)$, where we have \eqref{eq-biai2}
to relate it to our results for $\mbox{Ai}(z)$~\footnote{In fact, if
the change of variable applied to \eqref{eq-ai1} to become
\eqref{eq-ai4} is used in \eqref{eq-ai3} or \eqref{eq-bi1}, it is
easy to see that either \eqref{eq-Aipw} or \eqref{eq-biai2} are
produced.}.
\end{proof}

The splitting of the integration path for $|\mbox{arg }z| > 2\pi/3$
is related to the discontinuity of the Stokes' multipliers as
described in \cite{boyd:99}. The transition from Figure 2c to Figure
2d illustrates such a discontinuity in a visual way. As also stated
in \cite{boyd:99}, the Stokes phenomenon can be analyzed from a
topological point of view: the steepest descent contour moves from
including only one saddle point to both of them at $|\mbox{arg }z| =
2\pi/3$, and then to split into two different paths, now homotopic
to $\mathcal{L}_{31}$ and $\mathcal{L}_{12}$. Equation
\eqref{eq-Aipw} reflects this splitting, and $-(e^{i 2\pi/3}\,
z)^{3/2}=z^{3/2}s$ allows us to recognize the presence of a positive
exponential after the splitting. However, Theorem
\ref{th-airyseries} goes beyond \cite{boyd:99} in this sense and
shows the oscillatory behavior of the Airy function for the
anti-Stokes line for small, non-asymptotic values of $|z|$.

Theorem \ref{th-airyseries} produces a Hadamard expansion of a type
that is more general than those presented
in~\cite{paris:04,parisb:04}. It also reveals that the method
originally devised by Paris in them inherently contains a degree of
complexity when dealing with the branch points of the involved
complex algebraic curves that has not been studied before.

\section{Numerical analysis of the new series expansion and its comparison to Maclaurin and the asymptotic
expansions}
\label{sec-numan}

\begin{figure}

\begin{tabular}{c c}
\resizebox{6cm}{!}{\includegraphics{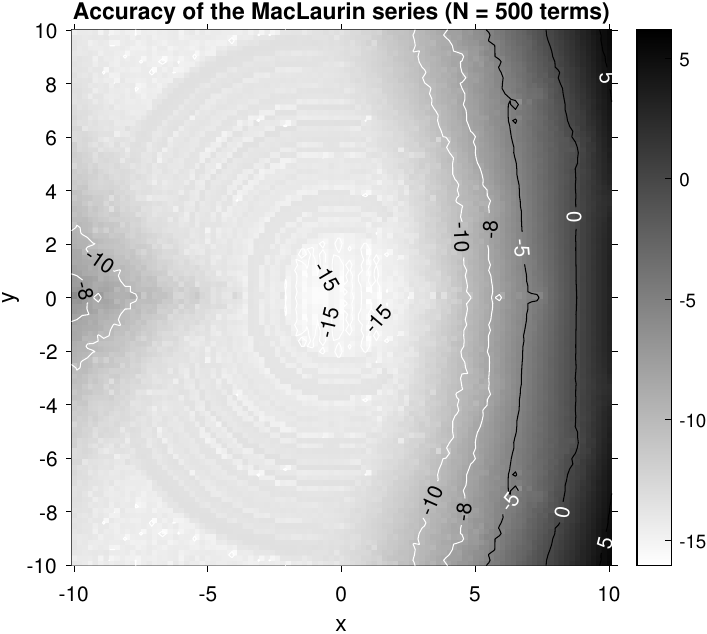}}
&
\resizebox{6cm}{!}{\includegraphics{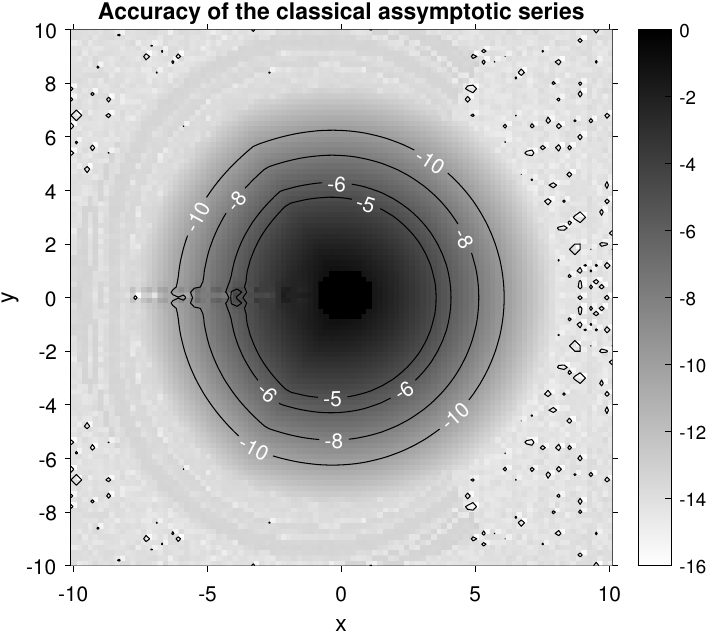}}
\vspace{-3pt} \\
{\scriptsize(a)} & {\scriptsize(b)} \\
\resizebox{6cm}{!}{\includegraphics{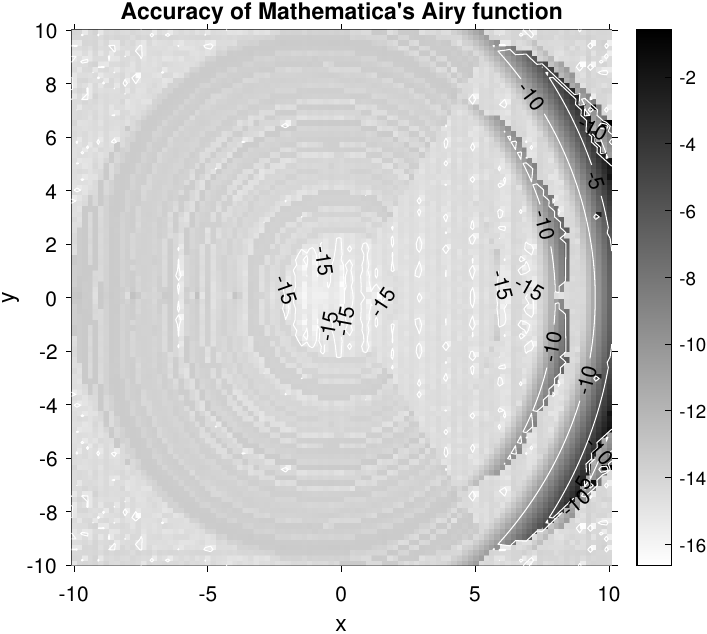}}
&
\resizebox{6cm}{!}{\includegraphics{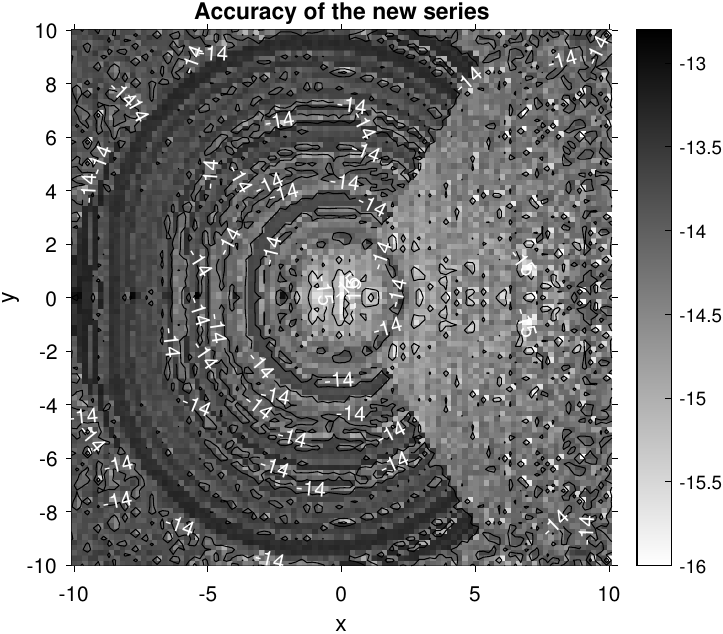}}
\vspace{-3pt} \\
{\scriptsize(c)} & {\scriptsize(d)}
\end{tabular}
\caption{The accuracy of the different expansions as defined in
\eqref{eq-accur} is shown in a), b) and d). The Maclaurin series is
computed with its first 500 terms. In c), the accuracy of
\textsc{Mathematica}$^{\circledR}$ with regard to
\textsc{Matlab}$^{\circledR}$ is shown.}
\label{fig:accur1}
\end{figure}

In this section, the new expansions of Theorem \ref{th-airyseries}
are analyzed numerically for $\{z=x + i y: -10\le x \le 10, -10\le y
\le 10\}$. This domain contains both the unit circle, where the
Maclaurin series is expected to perform best, and large enough
values of $|z|$ where the asymptotic series is applicable.  The
series of the new expansion are truncated when the difference
between two consecutive terms is less than the machine epsilon,
which, for the double precision floating-point format in use, is of
$2^{-53}$. 

The accuracy of any expansion method truncated at $n=N$ with respect
to the benchmark is defined as
\begin{equation}
\mbox{Accuracy}=\log_{10}\left|\frac{\left[\mbox{Ai}(z)\right]_{N}^{\scriptsize\mbox{method}}-\left[\mbox{Ai}(z)\right]^{\scriptsize\mbox{benchmark}}}
{\left[\mbox{Ai}(z)\right]^{\scriptsize\mbox{benchmark}}}\right|.
\label{eq-accur}
\end{equation}
As a benchmark, we use the routine for the integral of
$\mbox{Ai}(z)$ provided by \textsc{Matlab}$^{\circledR}$, which is
also fully compliant with both Scipy and Maxima for the
aforementioned domain. Figures \ref{fig:accur1}(a) to
\ref{fig:accur1}(d) show the accuracy of the Maclaurin series ($N =
500$), the classical asymptotic expansion,
\textsc{Mathematica}$^{\circledR}$'s Airy function and the new set
of expansions. The Maclaurin series has a bad performance for
$\mbox{Re}(z)> 5.5$, whereas the classical asymptotic series shows
this level of inaccuracy for $|z|\le 3.5$. The asymptotic series is
truncated when the error given by the first neglected term starts
rising. The values for the function given by \textsc{Mathematica
9.0}$^{\circledR}$ are also compared with the benchmark and some
differences are found, due to a different implementation. The new
expansion has an accuracy better than $10^{-12.78}$ for the complex
plane zone under study.

The number of terms needed for each one of the series expansions of
Theorem \ref{th-airyseries} in the numerical test of this section
are shown in Figure \ref{fig:histogramterms}. The convergence of the
expansions for $I_{\mathcal{L}_{s}}$,
$I_{\mathcal{L}_{\alpha_{0}^{+}}}$ and
$I_{\mathcal{L}_{\alpha_{0}^{-}}}$, as defined above, occurs for a
statistical mode of 20 terms. The largest number of terms is
required in the case of the computation of the sum
$I_{\mathcal{L}_{-\infty}}+I_{\mathcal{L}_{\infty}}$. The
coefficients $a_{n}$ and $b_{n}$ in equation \eqref{eq-I3c} can
computed from equation \eqref{eq-a3mb3m} once and for all values of
the variable $z$.

\begin{figure}

\resizebox{10cm}{!}{\includegraphics{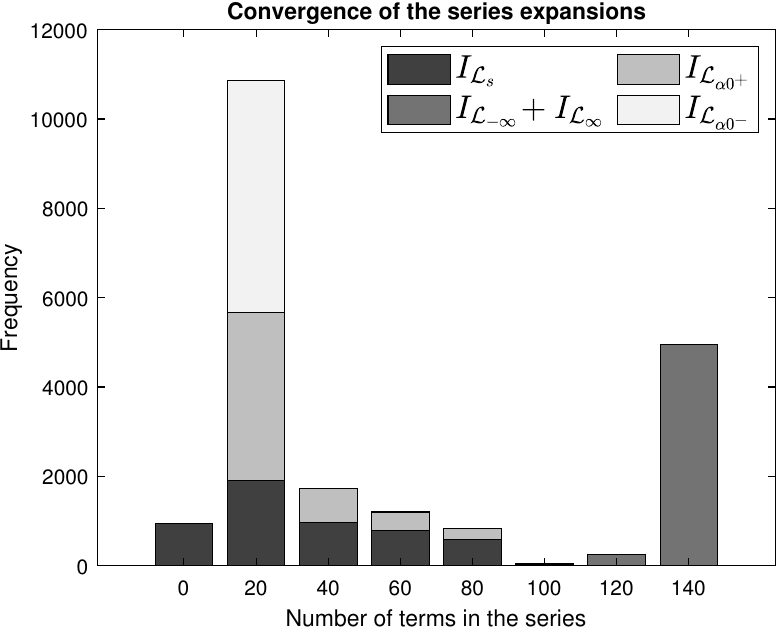}}
\caption{Histogram of the number of terms that are required by each
one of the series expansiones in equation \eqref{eq-aiIexp} to reach
convergence within the machine precision for double precision
floating-point numbers.}
\label{fig:histogramterms}
\end{figure}

\section{Conclusions}
A new convergent series expansion has been obtained for the Airy
function $\mbox{Ai}(z)$ of complex argument, and also for
$\mbox{Bi}(z)$ as a consequence of equation \eqref{eq-biai2}. This
new expansion includes incomplete Gamma functions of $|z|$ in its
five series. In this respect, they are a type of Hadamard expansions
as defined in~\cite{paris:04,parisb:04}. However, the new series are
different to the ones introduced in these references in the sense
that they include upper incomplete Gamma functions, and also in the
manner of selecting the points around which the series are computed.
The developments that are performed in the current analysis show the
complexity of the original idea of segmenting the steepest-descent
path in pieces where uniform convergence occurs and where the
integration is done avoiding an extreme proximity to the convergence
disk boundaries. The use of Puiseux series in this context for
expansions around the branch points is an additional new ingredient
of the treatment done here, as it is the fact that only a small and
finite number of segments is necessary for a full and exact spliting
of the integration path. These recourses had not been studied
in~\cite{paris:04,parisb:04}. The new approach could be extended to
other Laplace-type integrals by using the same procedure.

Theorem \ref{th-airyseries} also provides a clear picture of the
Stokes phenomenon present in the classical asymptotic series. It
shows the impact of the splitting of the steepest descent path in
the series of equation \eqref{eq-th1}, and the correspondence with
the well known property given in \eqref{eq-Aipw}. 

The convergence of the new expansions has also been studied in
section \ref{sec-numan}. It produces a level of accuracy that
improves the performance of the asymptotic expansion and equals that
of the Maclaurin series. Therefore, the new expansion is a candidate
to replace under a single framework the computation algorithm for
the Airy functions, which is currently based on the combined use of
the Maclaurin series, the asymptotic expansion and usually a
Gauss-Laguerre quadrature method for the corresponding integral
where the other two series are
inadequate~\cite{gilseguratemme:2007}.

\end{document}